\documentclass[12pt]{article}

\usepackage{mathtools} 
\usepackage{hyperref}

\DeclareMathOperator{\arccosh}{arccosh}
\DeclareMathOperator{\arccot}{arccot}

\usepackage[svgnames]{xcolor}
\definecolor{color1}{RGB}{27,158,119}
\definecolor{color2}{RGB}{217,95,2}
\definecolor{color3}{RGB}{117,112,179}
\definecolor{color4}{RGB}{231,41,138}

\usepackage[normalem]{ulem}%for \sout

\newcommand{\ve}{\varepsilon}

\usepackage{tikz}
\usepackage{graphicx}  % Add graphics capabilities
\usepackage{amsmath}  % Better maths support & more symbols
\usepackage{amssymb,amsthm}
\usepackage{color}
\numberwithin{equation}{section}

\newtheorem{theorem}{Theorem}
\newtheorem{lemma}{Lemma}

\newtheorem{proposition}{Proposition}
\newtheorem{corollary}{Corollary}
\newtheorem{remark}{Remark}

\title{\bf Selection of the ground state \\on a compact metric graph}

\author{Robert Marangell$^{1}$ and Dmitry E. Pelinovsky$^{2}$ \\
{\small $^{1}$  School of Mathematics and Statistics F07, University of Sydney, NSW 2006, Australia} \\
{\small $^{2}$ Department of Mathematics, McMaster University, Hamilton, Ontario, L8S 4K1, Canada} }

\begin{document}

\maketitle

\begin{abstract} 
We show that the ground state in the Fisher--KPP model on a compact metric graph with Dirichlet conditions on boundary vertices is either trivial (zero) or 
nontrivial and strictly positive. For positive initial data, we prove that the trivial ground state is globally asymptotically stable if the edges of the metric graph are uniformly small and the nontrivial ground state is globally asymptotically stable if the edges are uniformly large. 
For the intermediate case, we find a sharp criterion for the existence, uniqueness and global asymptotic stability of the trivial versus nontrivial ground state. Besides standard methods based on the comparison theory, energy minimizers, and the lowest eigenvalue of the graph Laplacian, we develop a novel method based on the period function for differential equations to characterize the nontrivial ground state in the particular case of flower graphs.
\end{abstract}

%\tableofcontents

\section{Introduction}
\label{sec-1}

Let $\Gamma$ be a compact metric graph with finitely many edges of finite lengths. We assume that 
$\Gamma$ is connected and there is at least one pendant (an edge with a boundary vertex not connected to any other edge). 
Let $u(t,x) : [0,\infty) \times \Gamma \to \mathbb{R}$ be the state variable satisfying the initial-value problem for the normalized Fisher--KPP equation 
\begin{equation}
\label{Fisher}
\left\{ \begin{array}{l} u_t = \Delta_{\Gamma} u + u(1-u), \quad t > 0, \\
u|_{t = 0} = u_0,\end{array} \right.
\end{equation}
where $\Delta_{\Gamma}$ is the Laplacian operator defined pointwisely on each edge 
of the graph $\Gamma$ subject to the suitable conditions at vertices 
and the quadratic nonlinearity $u(1-u)$ is also defined poinwisely on each edge of the graph. 

We use Dirichlet conditions (i.e. $u=0$) on the boundary vertices of the pendant(s). For all other (interior) vertices of the graph $\Gamma$, 
we use the Neumann--Kirchhoff (NK) conditions (continuity of $u$ and the zero sum of outward normal derivatives of $u$ 
to each edge connecting the vertex). The graph Laplacian $\Delta_{\Gamma}$ is a self-adjoint operator in $L^2(\Gamma)$ with the domain $D(\Delta_{\Gamma})$ defined in $H^2(\Gamma)$ pointwisely subject to the boundary conditions above. 

\subsection{Motivations}

Parabolic models on the metric graphs such as the Fisher--KPP, Keller--Segel, and chemotaxis models 
have been considered in \cite{Borsche2014,Bretti2014,Natalini2015,Camilli2017,JPS2019,DLP20} due to their 
applications to river networks, optical fibers, and other bio-engineering systems. Analysis 
of well-posedness of the Keller--Segel model on a compact metric graph was developed 
in \cite{Shemtaga2025} based on the previous study of the heat kernel estimates in \cite{Becker2021,Harrell2023}. 

Existence and stability of nonconstant states on compact metric graphs have been studied 
in \cite{Y2001,IK2021,I2022,MM25} for various models of mathematical biology. 
Bifurcations and asymptotic stability of the constant states 
in the Keller--Segel model, which includes the Fisher--KPP model, were recently studied in \cite{Shemtaga2024} 
subject to the NK conditions. Due to the Dirichlet conditions at the open vertices, 
the set of steady states on the metric graph $\Gamma$ considered here is more complicated than 
in \cite{Shemtaga2024} and includes both the trivial (zero) states and the nontrivial (positive and nonconstant) states.

Spreading speeds of a propagation front were also studied on unbounded metric graphs. 
It was shown in \cite{FHT2021} that the spreading speed may be slower 
than the limiting speed of the homogeneous Fisher-KPP model if the metric graph is an infinite random tree. Furthermore, front propagation may be blocked 
by the steady states pinned to the vertices of the graphs. This phenomenon was shown 
in a bistable reaction--difussion system for the star graphs in \cite{JM2019,JM2021} and 
for the tree graphs in \cite{JM2024,LM24}. Generalized traveling waves of the Fisher-KPP model on infinite metric graphs 
were recently considered in \cite{Shemtaga2026}. Our work on the compact metric graphs 
does not cover the propagation fronts, but we point out that the concept of traveling waves can be 
introduced for the Fisher--KPP, Keller--Segel and other reaction--diffusion models on the periodic (unbounded) metric graphs as 
is done for the nonlinear Schr\"{o}dinger (NLS) model in \cite{LeCoz2025}.

The purpose of this work is to find the precise conditions on the existence and asymptotic 
stability of the nontrivial (strictly positive) state of the Fisher--KPP model 
on the compact metric graph $\Gamma$. Besides the standard tools based 
on the comparison principle, energy minimizers, and the lowest eigenvalue
of the graph Laplacian, we also use the period function for periodic orbits to characterize 
all nontrivial states in the class of flower graphs with multiple loops. 

The period function was pioneered in \cite{KMPX21,NP2020} (see review in \cite{KNP2022}) 
in the context of the NLS model with the cubic nonlinearity. It has been used to study multiple 
(positive and nonpositive) steady states of the NLS model in the tadpole and other looping edge graphs 
in \cite{Pava2025,Pava2024,ACT24,ACT25}. However, the Fisher--KPP model has a quadradic nonlinearity, for which new estimates 
on the period function are needed. 

One of the technical novelties of our paper is the estimates on the period function to control periodic orbits near the homoclinic 
orbits. The method of \cite{Berkolaiko} to study the Dirichlet--to--Neumann map 
for the periodic orbits near the homoclinic orbits relies on the asymptotic expansion 
of the elliptic functions near the hyperbolic functions, and we show that a much simpler 
study can be developed by analyzing weakly singular integrals arising in the definition of the period function. We have also checked 
that the leading-order corrections in the asymptotic expansions of the elliptic functions 
cancel out for the quadratic nonlinearity, so that the method based on analysis of the weakly singular integrals 
is the only realistic tool to estimate the Dirichlet--to--Neumann map in the Fisher--KPP model. 

The new method based on the period function recovers the conclusions based on the classical methods. 
Moreover, it is applicable to a more general class of the reaction-diffusion models for which the 
classical methods might not be applicable. As a drawback, the period function can only be introduced 
for the second-order differential equations and the method is dependent on the structure of the graph $\Gamma$.

\subsection{Main results}

Since the Fisher--KPP equation is a gradient system, we have the free energy 
\begin{equation}
\label{energy}
H(u) = \frac{1}{2} \int_{\Gamma} [(\nabla_{\Gamma} u)^2 - u^2] dx + \frac{1}{3} \int_{\Gamma} u^3 dx, 
\end{equation}
such that $u_t = -\frac{\delta H}{\delta u}$. The free energy $H(u)$ is well defined for $u \in H^1_0(\Gamma)$, 
where $H^1_0(\Gamma)$ is defined pointwisely on edges of $\Gamma$ with Dirichlet conditions at the boundary vertices 
and the continuity conditions at the interior vertices. 
The mapping $(0,\infty) \ni t \mapsto H(u(t,\cdot)) \in \mathbb{R}$ is decreasing along every solution 
$u(t,\cdot) \in H^1_0(\Gamma)$, $t \in (0,\infty)$ which exists if $u_0 := u(0,\cdot) \geq 0$, see 
Theorem \ref{th-well}. Hence, we define the function space 
\begin{equation}
\label{function-space}
\mathcal{H}_0 = \left\{ u \in H^1_0(\Gamma) : \quad u \geq 0 \right\}
\end{equation}
and consider the minimization of the free energy $H(u)$ given by (\ref{energy}) on $\mathcal{H}_0$. 
The Euler--Lagrange equation for the critical points of $H(u)$ is given by the elliptic equation
\begin{equation}
\label{EL}
-\Delta_{\Gamma} u = u(1-u), \quad u \in D(\Delta_{\Gamma}),
\end{equation}
where $D(\Gamma) \subset H^1_0(\Gamma)$ is defined in $H^2(\Gamma)$ pointwisely on edges of $\Gamma$ subject to 
the boundary conditions of $H^1_0(\Gamma)$ and the additional conditions on the sum of outward normal derivatives being zero  
at the interior vertices of the graph $\Gamma$. Since $\Gamma$ consists of line segments in one spatial dimension 
and the NK conditions are natural boundary conditions for minimization of $H(u)$, 
every weak solution of $-\Delta_{\Gamma} u = u(1-u)$ in $H^1_0(\Gamma)$ is a strong solution of the elliptic 
equation (\ref{EL}) and vice versa. We define {\em steady states} as solutions of the Euler--Lagrange equation (\ref{EL}) 
and {\em ground states} as minimizers of the energy $H(u)$.

Let $\{ L_j \}$ be the set of lengths of edges of the graph $\Gamma$. The main results of this work are summarized as follows.

\begin{enumerate}
\item There is a unique global attractor of the initial-value problem \eqref{Fisher} for every $u_0 \in \mathcal{H}_0$. 
See Theorems \ref{th-well} and \ref{th-unique}. The attractor is either trivial or nontrivial depending whether 
the principal eigenvalue $\lambda_0(\Gamma)$ of $-\Delta_{\Gamma}$ in $L^2(\Gamma)$ is greater or smaller than one. 

	\item There exists $L_0 > 0$ such that if $\max_{j} L_j < L_0$, then the trivial (zero) ground state is a global attractor of the dynamics for initial data $u_0 \in \mathcal{H}_0$. The zero ground state is a minimizer of the energy $H(u)$ in $\mathcal{H}_0$ at the zero level. See Corallary \ref{cor-var-1}.
	
	\item For every $j$-th edge of length $L_j$, there is $L_j^{(0)} \in [0,\infty)$ such that for every $L_j > L_j^{(0)}$,  a strictly positive ground state (vanishing only at the boundary vertices due to Dirichlet conditions) is a global attractor of the dynamics for initial data  $u_0 \in \mathcal{H}_0 \backslash \{0\}$. The strictly positive ground state is a minimizer of the energy $H(u)$ in $\mathcal{H}_0$ at a negative level. See Corollary \ref{cor-var-2}. 
	
	\item There exist $L_* > L_0$ such that if $L_{\rm min} := \min_j L_j > L_*$, then the unique, strictly positive gound state $u_*$ satisfies
\begin{equation}
    \label{proximity}
	\| u_* - 1\|_{L^{\infty}(\Gamma_0)} \leq C e^{-L_{\rm min}},
\end{equation}
	where $C$ is a positive constant and $\Gamma_0$ is the part of $\Gamma$ without the pendants. See Theorem \ref{theorem-localized-state}.

    \item Consider a particular flower graph shown in Figure \ref{fig:flowers}, which consists of a line segment $[0,L]$ connected to $N$ loops 
    of different lengths $[0,L_1]$, $[0,L_2]$, $\dots$, $[0,L_N]$ with a Dirichlet condition at the boundary vertex and the NK condition 
    at the interior vertex. For every $L > L_* = L_*(L_1,L_2,\dots,L_N)$ given by 
\begin{equation}
    \label{threshold}
L_*(L_1,L_2,\dots,L_N) = \arccot \left( 2 \sum_{j=1}^N \tan(L_j) \right).
\end{equation}
there exists a unique, strictly positive ground state. See Theorems \ref{theorem-example-1}, \ref{theorem-example-2}, and \ref{theorem-example-3}
for the interval $[0,L]$ (with $L_1 = L_2 = \dots = L_N = 0$), the symmetric flower graph (with $L_1 = L_2 = \dots = L_N = L_0 > 0$), 
and the general flower graph, respectively. 
\end{enumerate}

\begin{figure}[htb!]
    \centering
    \includegraphics[width=0.3\linewidth]{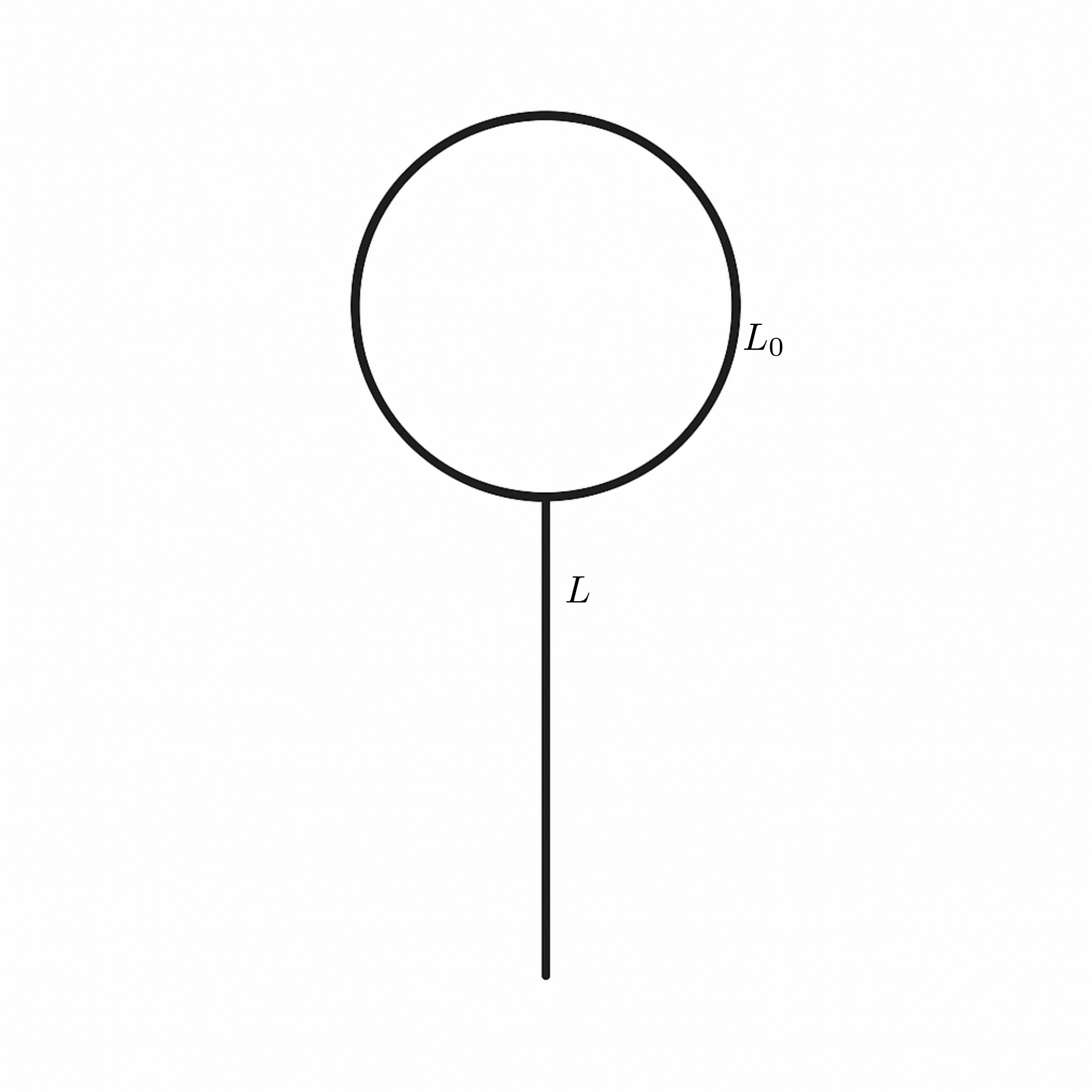}
    \includegraphics[width=0.3\linewidth]{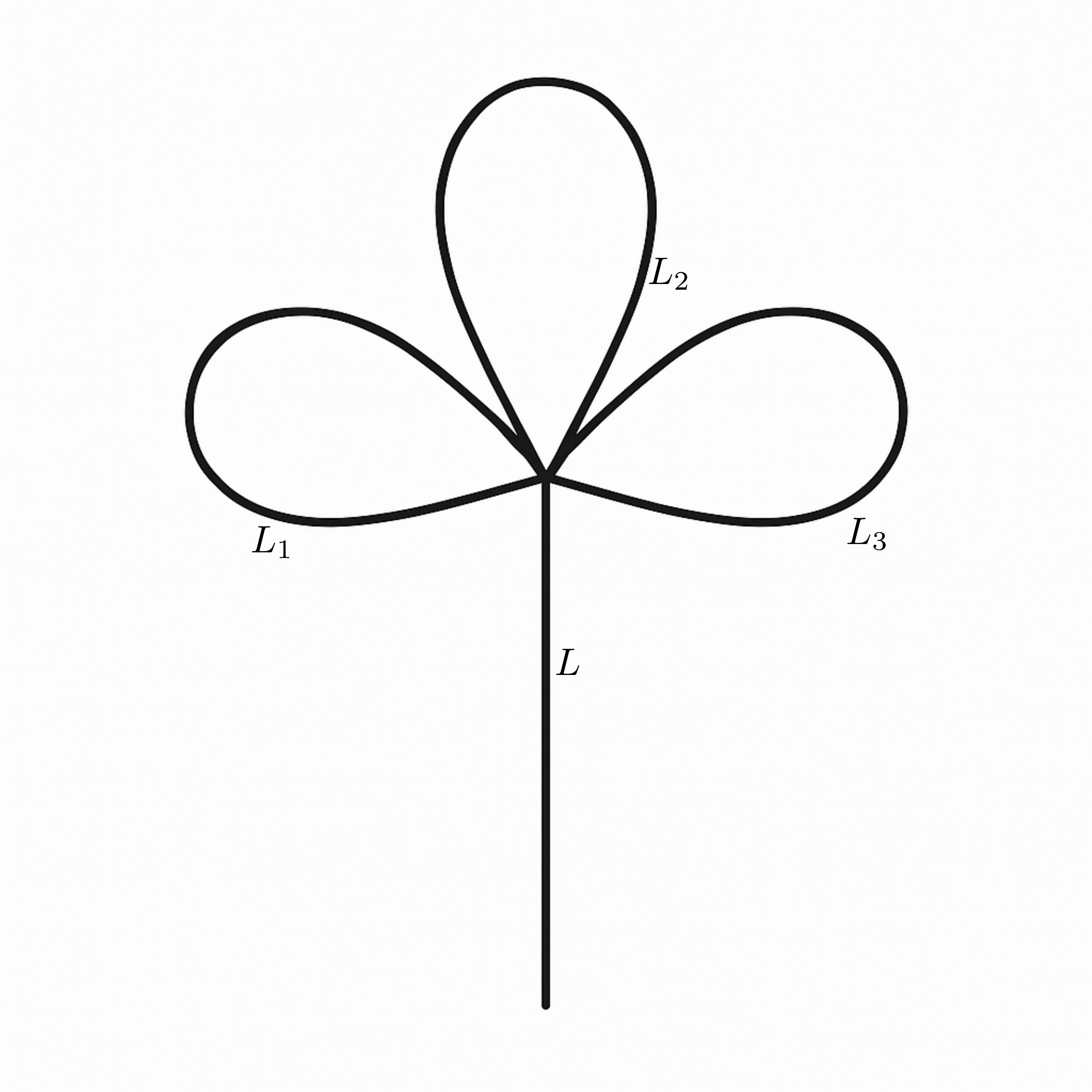}
    \includegraphics[width=0.3\linewidth]{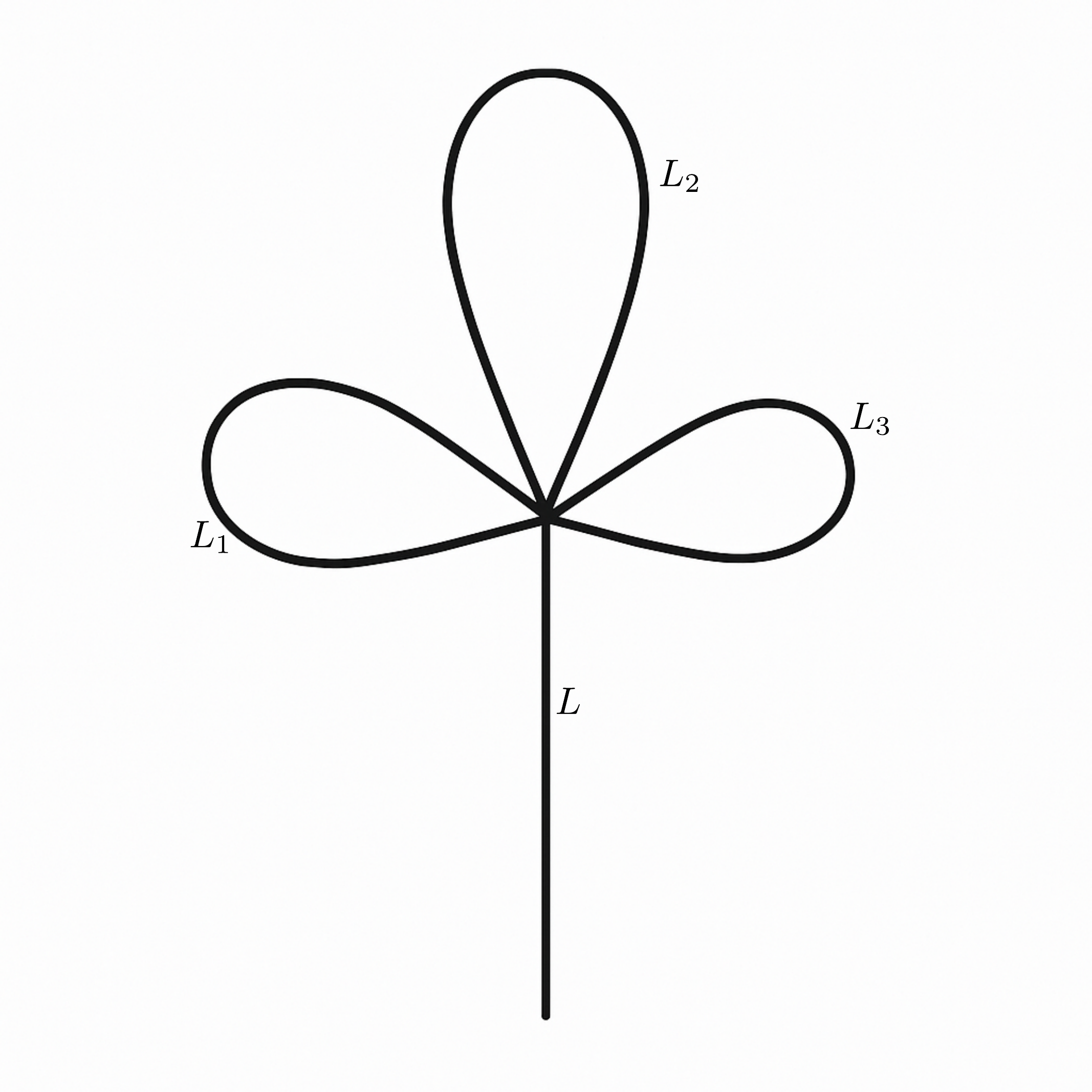}
    \caption{Left: A tadpole graph ($N=1$). Center: A flower graph with three loops of equal length ($N=3$). Right: A flower graph with three loops of unequal length ($N=3$).}
    \label{fig:flowers}
\end{figure}

\subsection{Organization of the paper}

We start in Section \ref{sec-eig} by studying how eigenvalues of $-\Delta_{\Gamma}$ in $L^2(\Gamma)$ depend on the edge lengths of the graph $\Gamma$. 
We show that every eigenvalue is monotonically decreasing of the length parameter for each edge of the graph $\Gamma$, see Lemma \ref{lem-linear-1}. 
Furthermore, eigenvalues of $-\Delta_{\Gamma}$ diverge to $\infty$ (converge to $0$) if the edge lengths shrink to $0$ (expand to $\infty$) uniformly, 
see Lemma \ref{lem-linear-2}. These results are relevant because the sharp criterion separating the existence of the trivial and nontrivial ground 
states is $\lambda_0(\Gamma) = 1$, where $\lambda_0(\Gamma)$ is the lowest eigenvalue of $-\Delta_{\Gamma}$ in $L^2(\Gamma)$.

In Section \ref{sec-min}, we use the comparison principle for elliptic equations to show that the range of the steady states is in $[0,1]$, 
see Lemma \ref{lem-comparison}. The sharp criterion of $\lambda_0(\Gamma) = 1$ separating the trivial minimizers of energy 
for $\lambda_0(\Gamma) \geq 1$ and the nontrivial minimizers of energy for $\lambda_0(\Gamma) \in (0,1)$ is shown in Lemma \ref{lem-var-1}. 
Furthermore, we show that the trivial ground state is the only steady state for $\lambda_0(\Gamma) \geq 1$ in Lemma \ref{lem-var-2} 
and that the nontrivial ground state is strictly positive on all points of $\Gamma$ except for the boundary vertices in Lemma \ref{lem-var-3}. By using the comparison principle for parabolic equations in Proposition \ref{prop-comparison} and the Sturm's theory for eigenvalues 
of the stationary Schr\"{o}dinger equation in Proposition \ref{lem:spectralcomp}, we give the proof of 
the global well-posedness in Theorem \ref{th-well} and the global asymptotic stability of the trivial and nontrivial 
ground states in Theorem \ref{th-unique}, respectively. Corollaries \ref{cor-var-1} and \ref{cor-var-2} show that the trivial ground state 
arises if the lengths of the graph $\Gamma$ are uniformly small and the nontrivial ground state arises if at least 
one edge of the graph $\Gamma$ is long.

In Section \ref{sec-4}, we introduce two period functions for the periodic orbits and give the monotonicity 
results for each of the period functions in Lemmas \ref{lem-period-1} and \ref{lem-period-2}. The asymptotic behavior 
of one of the two period functions near the homoclinic orbit is particularly important for graphs with long edges. 
It is obtained in Lemma \ref{lem-period-1-asymptotics}, where we avoid using elliptic functions and study the period function 
by using its definition as a weakly singular integral. The asymptotic behavior of both period functions 
near the center point is also important for the threshold criterion $\lambda_0(\Gamma) = 1$ and it is obtained in Lemma \ref{lem-period-3}.
As an application of Lemma \ref{lem-period-1-asymptotics}, we provide an alternate proof of the existence of 
a unique, strictly positive ground state in the limit of long graphs in Theorem \ref{theorem-localized-state}, 
where we also demonstrate the exponential smallness of the ground state satisfying the bound (\ref{proximity}).

Finally, a dynamical system approach is used in Section \ref{sec-5} to characterize the ground state for the flower graphs. We again apply 
the two period functions with the properties in Lemmas \ref{lem-period-1} and \ref{lem-period-2}
to obtain the existence of a unique, strictly positive ground state for the interval $[0,L]$ in Theorem \ref{theorem-example-1}, 
for the symmetric flower graph in Theorem \ref{theorem-example-2}, and for the general flower graph 
in Theorem \ref{theorem-example-3}. We also show in Proposition \ref{lem-lower-bound} 
that the threshold criterion $\lambda_0(\Gamma) = 1$ in Lemma \ref{lem-var-1} can be recovered alternatively 
by using the asymptotic behavior of the two period functions near the center point in Lemma \ref{lem-period-3},

\section{Eigenvalues of the graph Laplacian}
\label{sec-eig}

Since the graph $\Gamma$ is compact, the spectrum of $-\Delta_{\Gamma}$ in $L^2(\Gamma)$ consists of isolated eigenvalues (Theorem 3.1.1 in \cite{BK2013}). 
We show that each simple eigenvalue of $-\Delta_{\Gamma}$ is a monotonically decreasing function of each edge length and it is strictly monotone if the eigenfunction in the corresponding edge is nonzero. 

\begin{lemma}
	\label{lem-linear-1}
	Let the lengths $\{ L_j \}$ of edges $\{ e_j \}$ be fixed except for one edge parameterized as $[0,L]$. Let $\lambda(\Gamma) \geq 0$ be a simple eigenvalue of $-\Delta_{\Gamma}$ in $L^2(\Gamma)$ with the corresponding eigenfunction $\Psi : \Gamma \to \mathbb{R}$ in $D(\Delta_{\Gamma})$. Then, $\lambda(\Gamma)$ is a $C^1$ function of $L$ satisfying 
	\begin{equation}
	\label{slope}
	\frac{d \lambda(\Gamma)}{d L} = - |\psi'(L)|^2 - \lambda(\Gamma) |\psi(L)|^2 \leq 0,
	\end{equation}
	where $\psi(x) : [0,L] \to \mathbb{R}$ is the component of the eigenfunction 
	$\Psi$ on the edge $[0,L]$.
\end{lemma}

\begin{proof}
	Since $\lambda(\Gamma)$ is a simple eigenvalue of $-\Delta_{\Gamma}$ in $L^2(\Gamma)$ with the corresponding eigenfunction $\Psi : \Gamma \to \mathbb{R}$ in $D(\Delta_{\Gamma})$, we have 
	\begin{equation}
	\label{Rayleigh-0}
	\lambda(\Gamma) = \int_{\Gamma} |\nabla_{\Gamma} \Psi|^2 dx, \quad {\rm if} \;\; \int_{\Gamma} |\Psi|^2 dx = 1.
	\end{equation}
	By perturbation theory, if $\lambda(\Gamma)$ is a simple eigenvalue, then $\lambda(\Gamma)$ and $\Psi : \Gamma \to \mathbb{R}$ are $C^1$ functions of parameter $L$. It follows from the Neumann--Kirchhoff conditions at the vertex $v_L$ which corresponds to $x = L$ of the edge $[0,L]$ that 
	\begin{equation}
	\label{NK}
	\left\{ \begin{array}{l} 
\psi(L) = \Psi(v_L), \\
\psi'(L) + \sum_{e \to v_L} \Psi'(v_L) = 0, 
\end{array} \right.
	\end{equation}
where $\Psi'(v_L)$ denotes the outward derivative at $v_L$ from $e \to v_L$. By differentiating (\ref{NK}), we get that $\frac{\partial \Psi}{\partial L}$ satisfies the same boundary conditions as $\Psi$ at every other vertex, whereas at the vertex $v_L$ it satisfies 
\begin{equation}
\label{der-1}
\psi'(L) + \frac{\partial \psi(L)}{\partial L} = \frac{\partial \Psi(v_L)}{\partial L}.
\end{equation}
By differentiating the $L^2$ constraint in (\ref{Rayleigh-0}), we get 
\begin{equation}
\label{der-2}
|\psi(L)|^2 + 2 \int_{\Gamma} \Psi \frac{\partial \Psi}{\partial L} dx = 0.
\end{equation}
By differentiating $\lambda(\Gamma)$ in (\ref{Rayleigh-0}) and integrating by parts, we get 
\begin{align}
\frac{d \lambda(\Gamma)}{d L}  &= |\psi'(L)|^2 + 2 \int_{\Gamma} \nabla_{\Gamma} \Psi \frac{\partial \nabla_{\Gamma} \Psi}{\partial L} dx  \notag \\
&=	|\psi'(L)|^2 + 2 \psi'(L) \frac{\partial \psi(L)}{\partial L} + 2 \sum_{e \to v_L} \Psi'(v_L) \frac{\partial \Psi(v_L)}{\partial L}
- 2 \int_{\Gamma} (\Delta_{\Gamma} \Psi) \frac{\partial \Psi}{\partial L} dx,
	\label{der-3}
\end{align}		
where the contributions from all other vertices disappear due to the boundary conditions. Substituting (\ref{der-1}) and (\ref{der-2}) into (\ref{der-3}) and using the second condition in (\ref{NK}), we obtain 
\begin{align*}
\frac{d \lambda(\Gamma)}{d L}  &=	|\psi'(L)|^2 + 2 \psi'(L) \frac{\partial \psi(L)}{\partial L} + 2 \frac{\partial \Psi(v_L)}{\partial L} \sum_{e \to v_L} \Psi'(v_L) 
+ 2 \lambda(\Gamma) \int_{\Gamma}  \Psi \frac{\partial \Psi}{\partial L} dx \\
&=	-|\psi'(L)|^2 + 2 \frac{\partial \Psi(v_L)}{\partial L} \left[ \psi'(L) +  \sum_{e \to v_L} \Psi'(v_L) \right]
- \lambda(\Gamma) |\psi(L)|^2 \\
&= -|\psi'(L)|^2 - \lambda(\Gamma) |\psi(L)|^2.
\end{align*}
The right-hand side is negative since $\lambda(\Gamma) \geq 0$.	
\end{proof}

\begin{remark}
The statement of Lemma \ref{lem-linear-1} extends to the pendant $[0,L]$ with the Dirichlet condition at $x = L$. In this case, the derivative equation (\ref{slope}) holds with $\psi(L) = 0$ and the statement corresponds to Proposition 3.1.5 in \cite{BK2013}.
\end{remark}

%\begin{remark}
%	The statement of Lemma \ref{lem-linear-1} applies to the case of multiple eigenvalues. Since $-\Delta_{\Gamma}$ is a self-adjoint operator in $L^2(\Gamma)$, multiple eigenvalues are semi-simple and there exists invariant subspaces of $L^2(\Gamma)$ such that the multiple eigenvalues are continued in $L$ as $C^1$ functions. Then, one can use (\ref{Rayleigh-0}) and (\ref{NK}) for each invariant subspace with the same computations leading to (\ref{slope}). 
%\end{remark}

\begin{remark}
	\label{rem-positive}
	For the lowest eigenvalue $\lambda_0(\Gamma)$, the eigenfunction $\Psi : \Gamma \to \mathbb{R}$ is positive because it is obtained from the variational principle (Rayleigh quotient):
\begin{equation}
    \label{Rayleigh}
	\lambda_0(\Gamma) = \inf_{\Psi \in H^1_0(\Gamma)} \left\{ \int_{\Gamma} |\nabla_{\Gamma} \Psi|^2 dx : \quad  \int_{\Gamma} |\Psi|^2 dx = 1 
	\right\},
\end{equation}
	see Theorem 5.2.6 in \cite{BK2013}. Moreover, $\Psi$ may only vanish at the boundary vertices due to the Dirichlet conditions. Indeed, if it vanishes at an interior point of one edge and is not identically zero, then its derivative is nonzero at the same point and, hence, the eigenfunction changes the sign, a contradiction. If it vanishes at an interior vertex, then the zero sum of outward normal derivatives implies that the eigenfunction is negative at some edges connecting the interior vertex, a contradiction. Hence $\Psi(x) > 0$ for every $x \in \Gamma$ excluding the boundary vertices. As a result, 
	$|\psi'(L)|^2 + \lambda_0(\Gamma) |\psi(L)|^2 > 0$ and it follows from (\ref{slope}) that 
    $\lambda_0(\Gamma)$ is a strictly monotonically decreasing function of $L$.
\end{remark}

If all edge lengths are uniformly scaled by parameter $L$, then we get the following elementary result, in consistency with Lemma \ref{lem-linear-1}. 

\begin{lemma}
	\label{lem-linear-2} 
	Let the set of lengths $\{ L_j \}$ be given by $L_j = L \ell_j$ with $\{ \ell_j \}$ independent of $L > 0$. Then eigenvalues $\lambda(\Gamma)$ of $-\Delta_{\Gamma}$ are given by $\lambda(\Gamma) = L^{-2} \mu(\tilde{\Gamma})$, where $\mu(\tilde{\Gamma})$ are eigenvalues of $-\Delta_{\tilde{\Gamma}}$ for the rescaled graph $\tilde{\Gamma}$ with lengths $\{ \ell_j \}$.
\end{lemma}

\begin{proof}
	With appropriate parameterization of each edge $e_j$ of $\Gamma$ as $[0,L_j]$, we can use the scaling transformation $\Psi(x) = \Phi(y)$ with $y := \frac{x}{L}$ so that each edge $\tilde{e}_j$ of $\tilde{\Gamma}$ is now parameterized as $[0,\ell_j]$. By the chain rule, $-\Delta_{\Gamma} \Psi = \lambda(\Gamma) \Psi$ becomes $-L^{-2} \Delta_{\tilde{\Gamma}} \Phi = \lambda(\Gamma) \Phi$. 
	Hence $\mu(\tilde{\Gamma})= L^2 \lambda(\Gamma)$ satisfies 
 $-\Delta_{\tilde{\Gamma}} \Phi = \mu(\tilde{\Gamma}) \Phi$, which is independent of $L > 0$.
\end{proof}

\section{Asymptotically stable minimizers of energy}
\label{sec-min}

We consider minimizers of the energy $H(u)$ given by (\ref{energy}) in the function space $\mathcal{H}_0$ given by (\ref{function-space}). 
First, we show that the positivity constraint in $\mathcal{H}_0 \subset H^1_0(\Gamma)$ is compatible with solutions of 
the Euler--Lagrange equation (\ref{EL}). 

\begin{lemma}
    \label{lem-comparison}
Let $u \in D(\Delta_{\Gamma})$ be a solution to the elliptic equation (\ref{EL}). Then, it satisfies
    $$
    0 \leq u(x) \leq 1 \quad \mbox{\rm for every} \;\; x \in \Gamma.
    $$
\end{lemma}

\begin{proof}
    Solutions of the Laplace equation on metric graphs, $\Delta_{\Gamma} \varphi = 0$, $\varphi \in D(\Delta_{\Gamma})$, satisfy the same maximum principle as solutions of the Laplace equation in open regions of $\mathbb{R}^N$, see Appendix B in \cite{Berkolaiko}. Consequently, the comparison principle for the nonlinear elliptic equations extends from the open regions \cite[Section 13.2]{McOwen} to the compact connected metric graphs. In particular, if $f \in C^{\infty}(\mathbb{R})$ and if $u_{\pm} \in H^2(\Gamma)$ are solutions of 
\begin{equation}
    \label{EL-plus}
    \left\{ \begin{array}{l}    -\Delta_{\Gamma} u_+ + f(u_+) \geq 0 \;\;  \mbox{\rm in} \; \Gamma, \\
	u_+ \;\;\mbox{\rm satisfies NK conditions on interior vertices}, \\
	u_+ \geq 0 \;\; \mbox{\rm on boundary vertices},
	\end{array} \right.
\end{equation}
and 
\begin{equation}
    \label{EL-minus}
    \left\{ \begin{array}{l}    -\Delta_{\Gamma} u_- + f(u_-) \leq 0 \;\;  \mbox{\rm in} \; \Gamma, \\
	u_- \;\;\mbox{\rm satisfies NK conditions on interior vertices}, \\
	u_- \leq 0 \;\;  \mbox{\rm on boundary vertices},
	\end{array} \right.
\end{equation}
then a solution $u \in H^2(\Gamma)$ of 
\begin{equation}
    \label{EL-again}
    \left\{ \begin{array}{l}    -\Delta_{\Gamma} u + f(u) = 0 \;\; \mbox{\rm in} \; \Gamma, \\
	u \;\;\mbox{\rm satisfies NK conditions on interior vertices}, \\
	u = 0 \;\;  \mbox{\rm on boundary vertices},
	\end{array} \right.
\end{equation}
satisfies $u_- \leq u \leq u_+$ everywhere in $\Gamma$. Picking solutions $u_+ = 1$ and $u_- = 0$ of (\ref{EL-plus}) and (\ref{EL-minus}), respectively, with $f(u) := -u(1-u)$, proves the assertion for the solution of (\ref{EL-again}).
\end{proof}

\begin{remark}
    Among solutions of the Euler--Lagrange equation (\ref{EL}) in Lemma \ref{lem-comparison}, we distinguish  between the trivial (zero) solution and a nontrivial (positive) solution satisfying $u(x) > 0$ for at least some $x \in \Gamma$.
\end{remark}

The following lemma gives the explicit threshold criterion on the existence of a nontrivial minimizer of energy $H(u)$ in $\mathcal{H}_0$. The threshold criterion is given by $\lambda_0(\Gamma) = 1$, where $\lambda_0(\Gamma)$ is the lowest eigenvalue of $-\Delta_{\Gamma}$ in $L^2(\Gamma)$.

\begin{lemma}
	\label{lem-var-1}
	Let $\lambda_0(\Gamma)$ be the lowest eigenvalue of $-\Delta_{\Gamma}$ in $L^2(\Gamma)$. If $\lambda_0(\Gamma) \geq 1$, then the infimum of $H(u)$ in $\mathcal{H}_0$ is attained at $u = 0$ for which $H(0) = 0$. If $\lambda_0(\Gamma) \in (0,1)$, then the infimum of $H(u)$ in $\mathcal{H}_0$ is attained at $u = u_* \geq 0$ for which $H(u_*) < 0$.
\end{lemma}

\begin{proof}
Recall from the Rayleigh quotient (\ref{Rayleigh}) that 
$$
\int_{\Gamma} (\nabla_{\Gamma} u)^2 dx \geq \lambda_0(\Gamma) \| u \|^2_{L^2(\Gamma)}, \quad \forall u \in H^1_0(\Gamma).
$$
If $\lambda_0(\Gamma) \geq 1$, then we have for every $u \in \mathcal{H}_0$, 
\begin{align*}
H(u) &\geq \frac{1}{2} \int_{\Gamma} [(\nabla_{\Gamma} u)^2 - u^2] dx \\
&\geq \left( \lambda_0(\Gamma) - 1 \right) \| u \|^2_{L^2(\Gamma)},
\end{align*}
hence $H(u) \geq 0$ and $H(u) = 0$ is attained at $u = 0$ in $\Gamma$. If $\lambda_0(\Gamma) \in (0,1)$, then $H(u)$ is bounded from below 
in $\mathcal{H}_0$ by a negative level since 
\begin{align*}
H(u) &\geq -\frac{1}{2} \int_{\Gamma} u^2 dx + \frac{1}{3} \int_{\Gamma} u^3 dx \\
&\geq -\frac{1}{6} \sum_j L_j,
\end{align*}
where $\{ L_j \}$ are lengths of $\{ e_j \}$. Since $\Gamma$ is compact, there exists a minimizer $u_* \in \mathcal{H}_0$ which attains an infimum of $H(u)$. Since $u = 0$ is a saddle point of $H(u)$ if $\lambda_0(\Gamma) \in (0,1)$ due to the second derivative test, then we have $u_* \geq 0$ such that $u_*$ is not identically zero and $H(u_*) < 0 = H(0)$. 
\end{proof}

The following two lemmas give some refined results complementing Lemma \ref{lem-var-1}. The first one is to show that there are no other steady states 
of the elliptic equation (\ref{EL}) in $\mathcal{H}_0$ for $\lambda_0(\Gamma) \geq 1$. It is based on the 
contradiction argument (see, e.g., \cite{Hen80}). The second one is to ensure that the nontrivial ground state for $\lambda_0(\Gamma) \in (0,1)$ is strictly positive for all points of $\Gamma$ except for the boundary vertices.

\begin{lemma}
    \label{lem-var-2}
    The trivial ground state $u = 0$ is the unique steady state of the elliptic equation (\ref{EL}) in $\mathcal{H}_0$ 
    for $\lambda_0(\Gamma) \geq 1$.
\end{lemma}

\begin{proof}
Multiplying (\ref{EL}) by $u$ and integrating over $\Gamma$ by parts, we obtain 
\begin{align*}
\int_\Gamma u^3 dx &= \int_\Gamma u (\Delta_\Gamma u + u) dx \\
&= \int_{\Gamma} (u^2 - |\nabla_{\Gamma} u|^2 )dx \\
&\leq (1-\lambda_0(\Gamma)) \int_{\Gamma} u^2 dx \\
&\leq 0.
\end{align*}
Since $u \geq 0$ in $\mathcal{H}_0$ and $\int_{\Gamma} u^3 dx \leq 0$, then $u \equiv 0$ for every 
solution of (\ref{EL}) if $\lambda_0(\Gamma) \geq 1$. 
\end{proof}

\begin{lemma}
    \label{lem-var-3}
    The nontrivial ground state $u = u_*$ in $\mathcal{H}_0$ for $\lambda_0(\Gamma) \in (0,1)$ is strictly positive for all points of $\Gamma$ except for the boundary vertices.
\end{lemma}

\begin{proof}
Let $\varphi_0$ be the eigenfunction of $-\Delta_{\Gamma}$ for the lowest eigenvalue $\lambda_0(\Gamma)$. Then, $u \geq u_- := C \varphi_0$ by the comparison principle for the elliptic equations in the proof of Lemma \ref{lem-comparison} if $C := (1-\lambda_0(\Gamma))/\max_{x \in \Gamma} \varphi_0(x)$ since 
\begin{align*}
    \Delta_{\Gamma} u_- &= C \Delta_{\Gamma} \varphi_0 \\
    &= - C \lambda_0(\Gamma) \varphi_0 \\
    &= 
    - C \varphi_0 + C (1- \lambda_0(\Gamma)) \varphi_0  \\
    &\geq - C \varphi_0 + C^2 \varphi_0^2 \\
    &= -u_- + u_-^2,
\end{align*}
where we have used positivity of $\varphi_0$ on $\Gamma$. Since $u_- \in D(\Delta_{\Gamma})$ and $u_- > 0$ for all points of $\Gamma$ except for the boundary vertices, see Remark \ref{rem-positive}, it follows that $u_* \geq u_- > 0$ for all such points of $\Gamma$.
\end{proof}

We can now apply the previous results to study the selection of the ground state on a compact connected metric graph $\Gamma$ 
in the time evolution of the Fisher--KPP equation (\ref{Fisher}). For our analysis, we use the comparison principle for parabolic equations 
and the Sturm theory for stationary Schr\"{o}dinger equations which are widely known in open domains. 
We rewrite these two results for compact connected metric graphs in the following two propositions. 

\begin{proposition}
\label{prop-comparison}
Let $\Gamma$ be a compact connected metric graph. Assume that $f \in C^{\infty}(\mathbb{R})$ and $\underline{u}, \overline{u} \in C^0([0,\infty),H^2(\Gamma))$ are solutions of 
\begin{equation}
    \label{parabolic-plus}
    \left\{ \begin{array}{ll}    \overline{u}_t - \Delta_{\Gamma} \overline{u} + f(\overline{u}) \geq 0 \;\;  \mbox{\rm in} \; \Gamma, &\quad t > 0, \\
	\overline{u} \;\;\mbox{\rm satisfies NK conditions on interior vertices}, &\quad t > 0, \\
	\overline{u} \geq 0 \;\; \mbox{\rm on boundary vertices},  &\quad t > 0, \\
    \overline{u} |_{t = 0} \geq u_0, \;\;\mbox{\rm in} \; \Gamma &
	\end{array} \right.
\end{equation}
and 
\begin{equation}
    \label{parabolic-minus}
    \left\{ \begin{array}{ll}   \underline{u}_t -\Delta_{\Gamma} \underline{u} + f(\underline{u}) \leq 0 \;\;  \mbox{\rm in} \; \Gamma, &\quad t > 0,\\
	\underline{u} \;\;\mbox{\rm satisfies NK conditions on interior vertices}, &\quad t > 0,\\
	\underline{u} \leq 0 \;\;  \mbox{\rm on boundary vertices}, &\quad t > 0,\\
    \underline{u} |_{t = 0} \leq u_0, \;\;\mbox{\rm in} \; \Gamma &
	\end{array} \right.
\end{equation}
Then, a solution $u \in C^0([0,\infty),H^2(\Gamma))$ of 
\begin{equation}
    \label{parabolic-again}
    \left\{ \begin{array}{ll}    u_t -\Delta_{\Gamma} u + f(u) = 0 \;\; \mbox{\rm in} \; \Gamma, &\quad t > 0,\\
	u \;\;\mbox{\rm satisfies NK conditions on interior vertices}, &\quad t > 0,\\
	u = 0 \;\;  \mbox{\rm on boundary vertices}, &\quad t > 0, \\
    u|_{t = 0} = u_0,\;\;\mbox{\rm in} \; \Gamma &
	\end{array} \right.
\end{equation}
satisfies $u_-(t,\cdot) \leq u(t,\cdot) \leq u(t,\cdot)$ on $\Gamma$ for every $t \geq 0$.
\end{proposition}

\begin{proof}
The comparison principle holds for metric graphs since solutions of the Laplace equation on metric graphs, $\Delta_{\Gamma} \varphi = 0$, $\varphi \in D(\Delta_{\Gamma})$, satisfy the same maximum principle as solutions of the Laplace equation in open regions of $\mathbb{R}^N$, see Appendix B in \cite{Berkolaiko}. See Section 11.1 in \cite{McOwen}. 
\end{proof}

\begin{proposition}
\label{lem:spectralcomp}
Let $\Gamma$ be a compact connected metric graph and $\mu_0(V)$ be the lowest eigenvalue of the stationary Schr\"{o}dinger equation 
\begin{equation}
    \label{Schr-eq}
-\Delta_\Gamma \Psi + V \Psi = \mu \Psi, \quad \Psi \in D(\Delta_{\Gamma}),
\end{equation}
where $V(x) : \Gamma \to \mathbb{R}$ is a bounded potential. If $V_1(x) \leq V_2(x)$ for all $x$ on the interior of $\Gamma$ 
and $V_1(x) < V_2(x)$ for some $x$ in an open set in $\Gamma$, then we have $\mu_0(V_1) < \mu_0(V_2)$.
\end{proposition}

\begin{proof}
We normalize $\| \Psi \|_{L^2(\Gamma)} = 1$ and use the Rayleigh quotient for the lowest eigenvalue. This yields
\begin{equation*}
\mu_0(V_1) = \inf_{\Psi \in H^1_0(\Gamma)} \langle (-\Delta_{\Gamma} + V_1) \Psi, \Psi \rangle < 
\inf_{\Psi \in H^1_0(\Gamma)} \langle (-\Delta_{\Gamma} + V_2) \Psi, \Psi \rangle =  \mu_0(V_2), 
\end{equation*}
where we have used the strict positivity of $\Psi \in D(\Delta_{\Gamma})$ for the lowest eigenvalue of the Schr\"{o}dinger equation 
(\ref{Schr-eq}).
\end{proof}

The first main result describes the global well-posedness and the existence of a global attractor for the Fisher--KPP equation (\ref{Fisher}) 
with the initial data in $\mathcal{H}_0$. It uses the arguments in Section 11.3 of \cite{McOwen} extended to compact connected metric graphs. 

\begin{theorem}
    \label{th-well}
For every initial data $u_0 \in \mathcal{H}_0$, there exists a unique solution $u \in C^0([0,\infty),\mathcal{H}_0)$ to the Fisher--KPP equation (\ref{Fisher}) which depends continuously on the initial data. Moreover, there is a $u_* \in \mathcal{H}_0$ such that $u(t,\cdot) \to u_*$ in $H^1_0(\Gamma)$ as $t \to +\infty$.
\end{theorem}

\begin{proof}
  The operator $e^{-t \Delta_{\Gamma}} : H^1_0(\Gamma) \to H^1_0(\Gamma)$ is bounded for every $t \geq 0$ and 
$H^1_0(\Gamma)$ is a Banach algebra with respect to pointwise multiplication. Since the Fisher--KPP equation is a semi-linear equation, 
the contraction mapping principle gives a unique local solution $u \in C^0([0,t_0],H^1_0(\Gamma))$ to the initial-value problem (\ref{Fisher}) for some $t_0 \in (0,\infty)$. Continuous dependence on the initial data follows from the contraction principle.

By the comparison principle in Proposition \ref{prop-comparison}, if $u_0 \in \mathcal{H}_0$, then $\underline{u} = 0$ is 
a subsolution of $u$ so that $u(t,\cdot) \in \mathcal{H}_0$ for every $t \in [0,t_0]$. On the other hand, 
by Sobolev embedding of $H^1_0(\Gamma)$ into $L^{\infty}(\Gamma)$, 
if $u_0 \in \mathcal{H}_0$, then there is $C > 0$ such that $u_0(x) \leq C$ for all $x \in \Gamma$.
Let $\overline{u}$ be a solution of the initial-value problem 
\begin{align}
\label{ode-ivp}
    \left\{ \begin{array}{ll} 
    \overline{u}_t = \overline{u} (1 - \overline{u}), \quad & t > 0, \\
    \overline{u}(0) = C, \quad & \end{array} \right.
\end{align}
By the comparison principle in Proposition \ref{prop-comparison}, $\overline{u}$ is a supersolution of $u$ so that 
\begin{equation}
    \label{supremum-norm}
u(t,x) \leq \overline{u}(t), \quad \mbox{\rm for all  } x \in \Gamma \;\; \mbox{\rm and} \;\; t \in [0,t_0].
\end{equation}
It follows from (\ref{ode-ivp}) for any $C > 0$ that there exists a unique function $\overline{u}(t) \in C^1([0,\infty))$ 
such that $\lim\limits_{t \to +\infty} \overline{u}(t) = 1$. Since the graph $\Gamma$ is compact, 
the $L^2(\Gamma)$ norm of the unique local solution $u \in C^0([0,t_0],\mathcal{H}_0)$ remains bounded 
in the limit $t \to t_0$. Furthermore, since the Fisher--KPP equation is also a gradient model, we have 
\begin{equation}
    \label{Lyapunov}
\frac{d}{dt} H(u(t,\cdot)) = - \int_{\Gamma} \left( \frac{\delta H}{\delta u} \right)^2 dx \leq 0, \quad t \in (0,t_0),
\end{equation}
Due to monotonicity  of $H(u(t,\cdot))$ in (\ref{Lyapunov}) and the bound 
\begin{equation}
    \label{bound-H1}
\| \nabla u(t,\cdot) \|^2_{L^2(\Gamma)} \leq \| u(t,\cdot) \|^2_{L^2(\Gamma)} + 2 H(u(t,\cdot)),
\end{equation}
the $H^1_0(\Gamma)$ norm of the unique local solution $u \in C^0([0,t_0],\mathcal{H}_0)$ remains bounded 
in the limit $t \to t_0$. Hence, the unique local solution in $\mathcal{H}_0$ can be extended 
globally for every $t \in [0,\infty)$ as $u \in C^0([0,\infty),\mathcal{H}_0)$. 

To show the existence of the global attractor $u_* = \lim\limits_{t \to +\infty} u(t,\cdot)$ in $\mathcal{H}_0$,
we use the fact that the $H^1_0(\Gamma)$ norm of the global solution $u \in C^0([0,\infty),\mathcal{H}_0)$ remains bounded 
as $t \to +\infty$ due to the global bound on $\| u(t,\cdot) \|_{L^2(\Gamma)}$ following from (\ref{supremum-norm}) with 
$\lim\limits_{t \to +\infty} \overline{u}(t) = 1$, monotonicity of $H(u(t,\cdot))$ in $t$, and the bound (\ref{bound-H1}). 
Since $H(u)$ is bounded from below in $\mathcal{H}_0$ by Lemma \ref{lem-var-1} and the graph $\Gamma$ is compact, 
we must have $u(t,\cdot) \to u_*$ in $\mathcal{H}_0$ as $t \to +\infty$, where
$u_*$ is a steady state solution to the elliptic equation \eqref{EL}.
\end{proof}

Theorem \ref{th-well} establishes the existence of a unique global solution $u \in C^0([0,\infty),\mathcal{H}_0)$ 
to the initial-value problem \eqref{Fisher} and the existence of a global attractor $u_* \in \mathcal{H}_0$. 
If there is a unique steady state in the positive solutions of the elliptic equation (\ref{EL}), 
then there exists a unique global attractor of the initial-value problem (\ref{Fisher}) for all initial data in $\mathcal{H}_0 \backslash \{0\}$, 
which coincides with the ground state in Lemma \ref{lem-var-1} for the global minimum of $H(u)$. 
If $\lambda_0(\Gamma) \geq 1$, the unique global attractor is trivial by Lemma \ref{lem-var-2}. 
If $\lambda_0(\Gamma) < 1$, the ground state of $H(u)$ is strictly positive by Lemma \ref{lem-var-3}, but we 
need to prove that the global attractor is unique to coincide with the ground state of $H(u)$.
The following theorem gives the result on uniqueness of the global attractor in the Fisher--KPP equation (\ref{Fisher}). 
The proof expands the arguments in Corollary 2.2 and Propositions 3.1 and 3.2 of \cite{CC04} for compact connected metric graphs.

\begin{theorem} 
\label{th-unique} 
Let $\lambda_0(\Gamma)$ be the lowest eigenvalue of $-\Delta_{\Gamma}$ in $L^2(\Gamma)$. If $\lambda_0(\Gamma) \geq 1$,
then $u = 0$ is a unique global attractor to the initial-value problem \eqref{Fisher} for every $u_0 \in \mathcal{H}_0$. 
If $\lambda_0(\Gamma) < 1$, then the strictly positive ground state $u_* \in \mathcal{H}_0$ of $H(u)$ is a unique global attractor 
to the initial-value problem \eqref{Fisher} for every $u_0 \in \mathcal{H}_0 \backslash \{0\}$. 
\end{theorem}

\begin{proof}
For $\lambda_0(\Gamma) \geq 1$, the result follows by Lemma \ref{lem-var-2} and Theorem \ref{th-well}.
For $\lambda_0(\Gamma) < 1$, we consider a solution $\underline{u}$ 
to the initial-value problem \eqref{Fisher} with initial condition $\underline{u}|_{t = 0} = \varepsilon \varphi_0$ for some $\varepsilon >0$, 
where $\varphi_0$ is a strictly positive 
eigenfunction of $-\Delta_{\Gamma}$ for the lowest eigenvalue $\lambda_0(\Gamma)$ in $L^2(\Gamma)$. 
Then, it follows from (\ref{energy}) that 
$$
H(\underline{u}|_{t = 0}) = \frac{1}{2} \left(\lambda_0(\Gamma) - 1\right) \varepsilon^2 \| \varphi_0 \|^2_{L^2(\Gamma)} 
+ \frac{1}{3} \varepsilon^3 \| \varphi_0 \|^3_{L^3(\Gamma)}
$$
and we can select a sufficiently small value of $\varepsilon > 0$ such that $H(\underline{u}|_{t = 0}) < 0$. 
It follows from the monotonicity (\ref{Lyapunov}) that $H(\underline{u}(t,\cdot)) \leq H(\underline{u}|_{t = 0}) < 0$ 
so that the attractor for $\underline{u}$, denoted as $\underline{u}_*$, is nontrivial, which exists by Theorem \ref{th-well}.
By the preservation of positivity, $\underline{u}_* \in \mathcal{H}_0$ and since $\underline{u}_*$ is a steady 
state of the elliptic problem (\ref{EL}), we have by Lemma \ref{lem-var-3} that $\underline{u}_*(x) > 0$ for all $x \in \Gamma$ except for the boundary vertices.

Let $u \in C^0([0,\infty),\mathcal{H}_0)$ be a solution to the initial-value problem \eqref{Fisher} 
and assume first that $u_0(x) > 0$ for every $x \in \Gamma$ except for the boundary vertices. 
Then, we can always choose $\varepsilon > 0$ so that $\varepsilon \varphi_0(x) \leq u_0(x)$ 
for all $x \in \Gamma$. By the comparison principle in Proposition \ref{prop-comparison}, $\underline{u}$ is 
a subsolution to $u$ so that $\underline{u}(t,x) \leq u(t,x)$ for all $(t,x) \in [0,\infty) \times \Gamma$. 
The attractor $u_* = \lim\limits_{t \to +\infty} u(t,\cdot)$ exists by Theorem \ref{th-well}
and it follows from the comparison principle that $\underline{u}_*(x) \leq u_*(x)$ for all $x \in \Gamma$. 
Now, we consider two Schr\"{o}dinger equations similar to (\ref{Schr-eq}):
$$
-\Delta_{\Gamma} \psi + ( \underline{u}_* - 1) \psi = 0, \quad \psi \in D(\Delta_{\Gamma}) 
$$
and 
$$
-\Delta_{\Gamma} \psi + (u_* - 1) \psi = 0, \quad \psi \in D(\Delta_{\Gamma}), 
$$
which are satisfied by $\psi = \underline{u}_*$ and $\psi = u_*$ respectively. 
Hence, $\mu_0(\underline{u}_*-1) = \mu_0(u_*-1) = 0$ for the lowest eigenvalue of the Schrodinger 
equation (\ref{Schr-eq}) associated with the potentials $V_1 = \overline{u}_* - 1$ and $V_2 = u_* - 1$. 
If $\underline{u}_*(x) < u_*(x)$ for some $x$ in an open set in $\Gamma$, 
then we get a contradiction with $\mu_0(\underline{u}_*-1) <\mu_0(u_*-1)$ following from Proposition \ref{lem:spectralcomp}.
This proves that $\underline{u}_*(x) = u_*(x)$ for all $x \in \Gamma$, hence the global attractor for 
every global solution $u \in C^0([0,\infty),\mathcal{H}_0)$ is unique as long as $u_0(x) > 0$ for every $x \in \Gamma$ except for the boundary vertices. 
 
Let $u \in C^0([0,\infty),\mathcal{H}_0)$ be a solution to the initial-value problem \eqref{Fisher} and assume now 
that $u_0 \in \mathcal{H}_0$ may vanish at some points of $\Gamma$ other than the boundary vertices but $u_0 \neq 0$ identically. 
Then we can advance the solution $u(t,\cdot)$ forward in time for some $t_0 > 0$, and by the strong maximum principle 
we have that the solution $u(t,\cdot)$ is positive everywhere except the boundary vertices. Then we can choose again 
$\ve > 0$ such that $\varepsilon \varphi_0(x) \leq u(t_0,x)$ for all $x \in \Gamma$ 
and use $\underline{u}(t-t_0,\cdot)$ with $\underline{u}(0,\cdot) = \varepsilon \varphi_0$ 
as a subsolution to $u(t-t_0,\cdot)$ for $t \geq t_0$. As a result, we again 
obtain a unique global attractor for every global solution $u \in C^0([0,\infty),\mathcal{H}_0)$ with $u_0 \in \mathcal{H}_0$ as long 
as $u_0 \neq 0$ identically.
\end{proof}

\begin{remark}
For $\lambda_0(\Gamma) > 1$, we can show that $u = 0$ is a unique global attractor for every $u_0 \in \mathcal{H}_0$ 
by using the comparison principle in Proposition \ref{prop-comparison}. To do so, we consider the function 
$\overline{u}(x,t) = C e^{(1 - \lambda_0(\Gamma)) t} \varphi_0(x)$, where $\varphi_0$ is a strictly positive 
eigenfunction of $-\Delta_{\Gamma}$ for the lowest eigenvalue $\lambda_0(\Gamma)$ in $L^2(\Gamma)$. 
We observe that 
\begin{equation}\label{eq:super}
\overline{u}_t = \left(1 - \lambda_0(\Gamma) \right) \overline{u} \geq \Delta_{\Gamma} \overline{u} + \overline{u} (1-\overline{u}).
\end{equation}
If $u \in C^0([0,\infty),\mathcal{H}_0)$ is any solution to the initial-value problem \eqref{Fisher}, 
then we can always choose a sufficient large value of $C$ such that $u_0(x) \leq C \varphi_0(x)$ for all $x \in \Gamma$. 
By the comparison principle in Proposition \ref{prop-comparison}, $\overline{u}$ is a supersolution 
to $u$ so that $u(t,x) \leq \overline{u}(t,x)$ for all $(t,x) \in [0,\infty) \times \Gamma$ and we get 
$$
\lim_{t \to \infty} u(t,x) \leq \lim_{t \to \infty} \overline{u}(t,x) = 0, \quad \mbox{\rm for every} \;\; x \in \Gamma.
$$
However, this argument is not sufficient for $\lambda_0(\Gamma) = 1$ since in this case, $\overline{u}$ is $t$-independent 
and only shows that $\| u(t,\cdot) \|_{L^{\infty}(\Gamma)}$ is globally bounded as $t \to +\infty$.
\end{remark}

Theorems \ref{th-well} and \ref{th-unique} reduce the existence of a nontrivial global attractor to determining whether or not the lowest eigenvalue $\lambda_0(\Gamma)$ is greater or less than $1$. This precise condition can be satisfied depending on the edge lengths of the graph $\Gamma$, 
as the following two corollaries demonstrate. 

\begin{corollary}
	\label{cor-var-1}
	There is $L_0 > 0$ such that if $\max_j L_j < L_0$, then $u = 0$ is a unique global attractor for every $u_0 \in \mathcal{H}_0$. 
\end{corollary}

\begin{proof}
By Lemma \ref{lem-linear-2}, $\lambda_0(\Gamma) \to +\infty$ as $L \to 0$ after rescaling $L_j = L \ell_j$ for $L$-independent $\{ \ell_j \}$. 
Then, there is $L_0 > 0$ such that $\lambda_0(\Gamma) > 1$ for $\max_j L_j < L_0$ and Theorem \ref{th-unique} applies.
\end{proof}

\begin{corollary}
	\label{cor-var-2}
	For every $j$-th edge of length $L_j$, there is $L_j^{(0)} \in [0,\infty)$ such that for every $L_j > L_j^{(0)}$ with this $j$, 
    the strictly positive ground state of $H(u)$ is a unique global attractor for every $u \in \mathcal{H}_0 \backslash \{0\}$. 
\end{corollary}

\begin{proof}
By Lemma \ref{lem-linear-1} and Remark \ref{rem-positive}, $\lambda_0(\Gamma)$ is strictly monotonically decreasing with respect to $L_j$. Moreover, $\lambda_0(\Gamma) \to 0$ as $L_j \to \infty$ because it is bounded from below by $0$ and it cannot converge to a positive level since the spectrum of $-\Delta_{\Gamma}$ with $L_j = \infty$ includes a continuous spectrum at $[0,\infty)$. Therefore, either $\lambda_0(\Gamma) \in (0,1]$ in the limit $L_j \to 0$ in which case $L_j^{(0)} = 0$ or $\lambda_0(\Gamma) > 1$ in the limit $L_j \to 0$ in which case $L_j^{(0)} > 0$ and $\lambda_0(\Gamma) \in (0,1)$ for $L_j > L_j^{(0)}$. In either case, Theorem \ref{th-unique} applies. 
\end{proof}

\section{Period functions for the positive ground state}
\label{sec-4}

The purpose of the period functions is to characterize the nontrivial (positive) ground state 
more precisely. The period functions for cubic differential equations on metric graphs were introduced in our previous work \cite{KMPX21}. Here we extend the method to the quadratic differential equations, where many details are different, including the asymptotic expansion of the period function 
in the limit of long graphs. 

Let us consider the integral curves of the second-order equation 
\begin{equation}
\label{ode}
u''(x) + u(x) - u(x)^2 = 0, \quad u(x) : \mathbb{R} \to \mathbb{R}.
\end{equation}
If $L_{\rm min} = \min_j L_j$ is large, the ground state $u_*$ is locally close to $1$ on $\Gamma_0$, 
where $\Gamma_0$ is the part of the metric graph $\Gamma$ without the pendants with the boundary vertices. 
Therefore, it makes sense to rewrite 
the second-order equation (\ref{ode}) for $u = 1 - \tilde{u}$ with 
\begin{equation}
\label{ode-tilde}
\tilde{u}''(x) - \tilde{u}(x) + \tilde{u}(x)^2 = 0, \quad \tilde{u}(x) : \mathbb{R} \to \mathbb{R}.
\end{equation}
In what follows, we introduce two period functions for the second-order equation (\ref{ode-tilde}). 

\subsection{The period function to the point $(p,q)$}

Let $(p,q) \in (0,1) \times (-\infty,0)$ be a point on the phase plane of the second-order equation (\ref{ode-tilde}) 
which admits a saddle point at $(0,0)$ and a center point at $(1,0)$. 
We are considering an integral curve which corresponds to the 
constant value of the first-order invariant
\begin{equation}
\label{invariant}
E(\tilde{u},\tilde{v}) = \tilde{v}^2 - \tilde{u}^2 + \frac{2}{3} \tilde{u}^3, \quad \tilde{v}:= \frac{d \tilde{u}}{dx},
\end{equation}
starting with the initial point $(\tilde{u},\tilde{v}) = (1,\tilde{q})$ and ending at a point $(\tilde{u},\tilde{v}) = (p,q)$, 
where $\tilde{q} = \tilde{q}(p,q) < 0$ is found from 
\begin{equation}
\label{energy-level}
E(p,q) = q^2 - p^2 + \frac{2}{3} p^3 = \tilde{q}^2 - \frac{1}{3} \quad \Rightarrow \quad 
\tilde{q}(p,q) = -\sqrt{q^2 + \frac{1}{3} - p^2 + \frac{2}{3} p^3}.
\end{equation}
The period function for the part of the integral curve connecting the two points is given by 
\begin{equation}
\label{period-function}
T(p,q) = \int_{p}^1 \frac{du}{v}, \quad v := \sqrt{E(p,q) + A(u)}, \quad A(u) := u^2 - \frac{2}{3} u^3.
\end{equation}
Figure \ref{fig:period-T-function} shows the part of the integral curve between the points $(1,\tilde{q})$ and $(p,q)$. 

\begin{figure}[htb!]
    \centering
    \includegraphics[width=8cm,height=5.5cm]{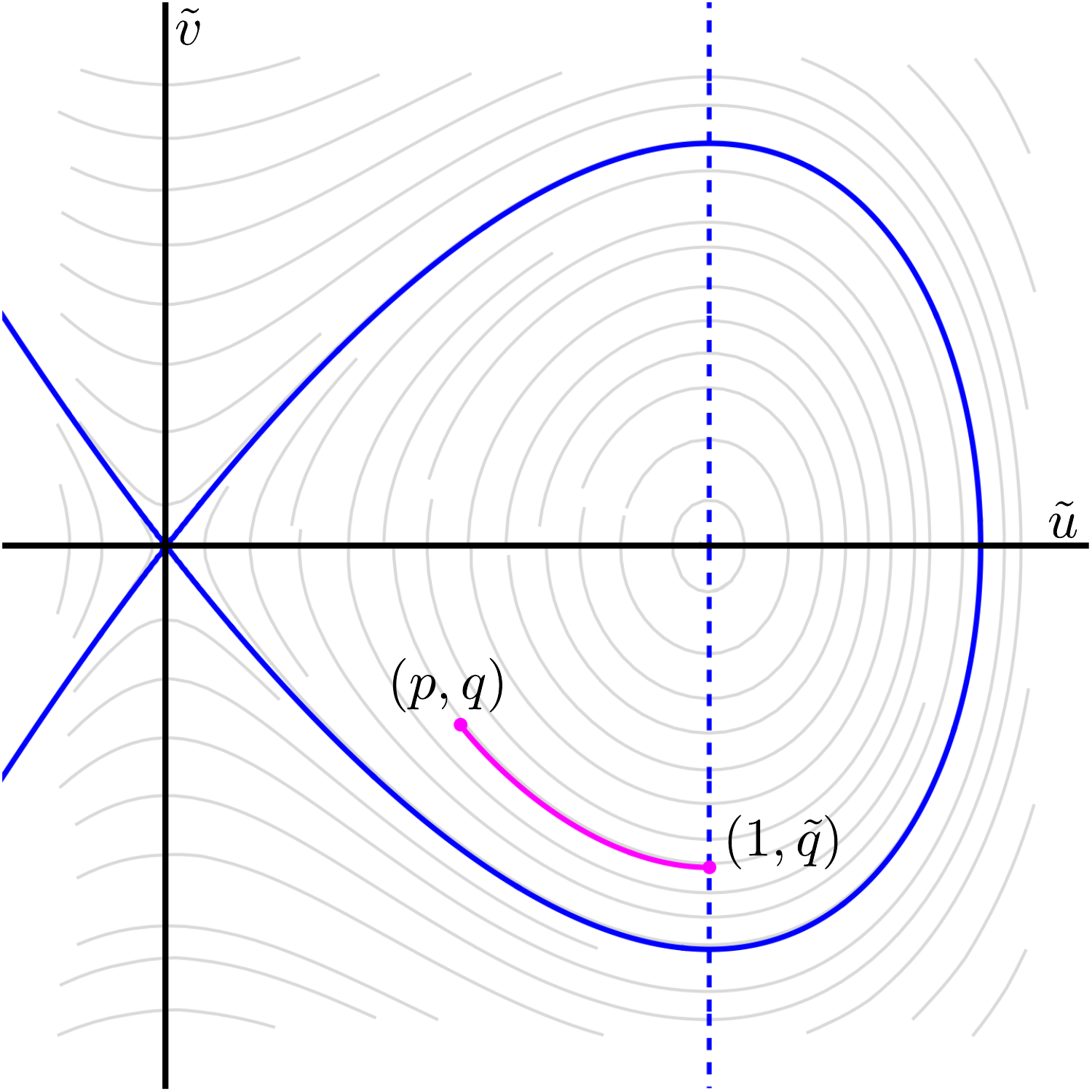}
    \caption{An example of the period function $T(p,q)$ given by (\ref{period-function}) in terms of the phase portrait 
    obtained from the level curves of $E(\tilde{u},\tilde{v})$ given by (\ref{invariant}). 
    The homoclinic orbit is in blue, the integral curve from $(1,\tilde{q})$ to $(p,q)$ is in magenta, 
    and the other orbits are in gray.}
    \label{fig:period-T-function}
\end{figure}

The following lemma gives the monotonicity of $T(p,q)$ with respect to two independent parameters $(p,q)$.

\begin{lemma}
	\label{lem-period-1}
	For every $(p,q) \in (0,1) \times (-\infty,0)$, we have 
	\begin{equation}
	\label{monotonicity}
	\frac{\partial T}{\partial p} < 0, \qquad 
	\frac{\partial T}{\partial q} > 0.
	\end{equation} 
\end{lemma}

\begin{proof}
	Since $p > 0$ and $q < 0$ for the part of the integral curve, we can differentiate $T(p,q)$ 
    by using (\ref{period-function}) directly and obtain 
	\begin{align}
	\label{der-p}
	\frac{\partial T}{\partial p} &= p(1-p) \int_p^1 \frac{du}{v^3} + \frac{1}{q}, \\
	\label{der-q}
	\frac{\partial T}{\partial q} &= -q \int_p^1 \frac{du}{v^3},
	\end{align}
where we have used 
\begin{equation}
	\label{derivative-E}
\frac{\partial E}{\partial p} = -2p(1-p), \qquad 
\frac{\partial E}{\partial q} = 2q,
\end{equation}
and $v |_{u = p} = -q$. Since $(p,q) \in (0,1) \times (-\infty,0)$, it follows from (\ref{der-q}) that $\frac{\partial T}{\partial q} > 0$, whereas the sign of $\frac{\partial T}{\partial p}$ is not conclusive from (\ref{der-p}). In order to complete the proof of (\ref{monotonicity}), we renormalize the weak singularity of the integrand in (\ref{period-function}) by using 
	\begin{align*}
	[E(p,q) + A(1)] T(p,q) &= \int_p^1 \frac{v^2 - [A(u)-A(1)]}{v} du \\
	&= \int_{p}^1 v du - \int_p^1 \frac{A(u) - A(1)}{v} du, 
	\end{align*}
	where $E(p,q) + A(1) > 0$ for every $E(p,q) \in (-\frac{1}{3},0)$ since $A(1) = \frac{1}{3}$. Along the integral curve $E(p,q) = E(u,v)$, we have by the chain rule
	\begin{align*}
	d \left[ \frac{A(u) - A(1)}{A'(u)} v \right] = \left(1 - \frac{A''(u) [A(u) - A(1)]}{[A'(u)]^2} \right) v du + \frac{A(u) - A(1)}{2 v} du,
	\end{align*}
	which yields 
	\begin{align*}
[E(p,q) + A(1)] T(p,q) = \int_{p}^1 \left(3 - \frac{2 A''(u) [A(u) - A(1)]}{[A'(u)]^2} \right) v du - 2 \frac{A(u) - A(1)}{A'(u)} v \biggr|_{u = p}^{u = 1}. 
\end{align*}
Since 
$$
A(u) - A(1) = - \frac{1}{3} (1-u)^2 (1+2u), \quad A'(u) = 2u (1-u),
$$	
the last term is zero at $u = 1$. At $u = p$, we get $v = -q$ which yields
	\begin{align*}
[E(p,q) + A(1)] T(p,q) = \int_{p}^1 \left(3 - \frac{2 A''(u) [A(u) - A(1)]}{[A'(u)]^2} \right) v du + \frac{(1-p)(1+2p)}{3p} q.
\end{align*}
The right-hand side is now continuously differentiable in $(p,q)$ from which we obtain 
	\begin{align}
[E(p,q) + A(1)] \frac{\partial T}{\partial p} &= -p(1-p) \int_{p}^1 \left(1 - \frac{2 A''(u) [A(u) - A(1)]}{[A'(u)]^2} \right) \frac{du}{v}  \notag\\
& \qquad 
+ q \left(3 - \frac{2 A''(p) [A(p) - A(1)]}{[A'(p)]^2} \right)
+ q \frac{d}{dp}  \frac{(1-p)(1+2p)}{3p} \notag \\
&= -p(1-p) \int_{p}^1 \frac{1-u^2}{3u^2 v} du + q 
\label{dTdp}
\end{align}
and 
	\begin{align}
[E(p,q) + A(1)] \frac{\partial T}{\partial q} &= q \int_{p}^1 \left(1 - \frac{2 A''(u) [A(u) - A(1)]}{[A'(u)]^2} \right) \frac{du}{v} + \frac{(1-p)(1+2p)}{3p} \notag \\
&= q \int_{p}^1 \frac{1-u^2}{3u^2 v} du + \frac{(1-p)(1+2p)}{3p},
\label{dTdq}
\end{align}
where we have used (\ref{derivative-E}) and  the definition of $T$ and $A(p)$ in (\ref{period-function}). 
Since $E(p,q) + A(1) > 0$ for $(p,q) \in (0,1) \times (-\infty,0)$, we have $\frac{\partial T}{\partial p} < 0$ from (\ref{dTdp}).
\end{proof}

The asymptotic behavior of $T(p,q)$ near $(0,0)$ is obtained in the following lemma. It is used in the proof of Theorem \ref{theorem-localized-state} below. 

\begin{lemma}
The period function $T(p,q)$ in Lemma \ref{lem-period-1} satisfies the limit
	\begin{equation}
	\label{limits-T}
	\displaystyle
\lim_{q = Qp, \; p \to 0^+} p \frac{\partial T}{\partial p} = \frac{-1}{1-Q} = - \lim_{q = Qp, \; p \to 0^+}  p \frac{\partial T}{\partial q},
	\end{equation}
	where $Q \in (-\infty,0)$ is fixed.  Moreover, $T(p,q)$ satisfies the asymptotic expansion
	\begin{equation}
\label{asymptotics-T}
T(p,q) = - \ln\left(\frac{p-q}{12}\right) - x_0 + \mathcal{O}(p), \quad 
\mbox{\rm as} \;\; p \to 0^+, \;\; q = Qp,
\end{equation} 
where $x_0 = 2 \arccosh\left(\frac{\sqrt{3}}{\sqrt{2}}\right)$.
\label{lem-period-1-asymptotics}
\end{lemma} 

\begin{proof}
To get the limits (\ref{limits-T}), we set $q = Q p$ with fixed $Q \in (-\infty,0)$, multiply the expression (\ref{dTdp}) by $3p$ and take the limit $p \to 0^+$:
	\begin{align*}
\lim_{p \to 0^+} p \frac{\partial T}{\partial p} &= - \lim_{p \to 0^+} p^2 \int_{p}^1 \frac{1-u^2}{u^2 \sqrt{(Q^2-1)p^2 + u^2 + \frac{2}{3} (p^3-u^3)}} du \\
&= - \lim_{p \to 0^+} \int_{1}^{1/p} \frac{1-p^2 x^2}{x^2 \sqrt{x^2 + Q^2 -1 + \frac{2}{3} p (1-x^3)}} dx,
\end{align*}
where we substituted $u = px$. The integrand is absolutely convergent for every $p > 0$ and the limiting integral is also convergent. Therefore, by Lebesque's Dominated Convergence Theorem, we obtain 
	\begin{align}
    \label{tech-expr-1}
\lim_{p \to 0^+} p \frac{\partial T}{\partial p} &= -  \int_{1}^{\infty} \frac{dx}{x^2 \sqrt{x^2 + Q^2 - 1}} dx = 
-\frac{1}{1-Q}, 
\end{align}
which yields the first limit in (\ref{limits-T}). Computations of the explicit integral are different between $Q \in (-\infty,-1)$, $Q = -1$, and $Q \in (-1,0)$ and are performed as follows
\begin{align*}
Q \in (-\infty,-1) : &\quad 
-  \int_{1}^{\infty} \frac{dx}{x^2 \sqrt{x^2 + Q^2 - 1}} = -\frac{1}{Q^2-1} \int_{{\rm arcsinh}(1/\sqrt{Q^2-1})}^{\infty} \frac{dy}{\sinh^2(y)} = \frac{1}{Q-1}, \\
Q = -1 : &\quad 
-  \int_{1}^{\infty} \frac{dx}{x^2 \sqrt{x^2}} = \frac{1}{2x^2} \biggr|^{x \to \infty}_{x = 1} = -\frac{1}{2}, \\
Q \in (-1,0) : &\quad 
-  \int_{1}^{\infty} \frac{dx}{x^2 \sqrt{x^2 - (1-Q^2)}} = -\frac{1}{1-Q^2} \int_{\arccosh(1/\sqrt{1-Q^2})}^{\infty} \frac{dy}{\cosh^2(y)} = -\frac{1}{1-Q}, 
\end{align*}
To get the second limit in (\ref{limits-T}), we obtain similarly from (\ref{dTdq}) that 
	\begin{align}
\lim_{p \to 0^+} p \frac{\partial T}{\partial q} &= Q \lim_{p \to 0^+} p^2 \int_{p}^1 \frac{1-u^2}{u^2 v} du + \lim_{p \to 0^+} (1-p)(1+2p) \notag \\
&= Q \int_{1}^{\infty} \frac{dx}{x^2 \sqrt{x^2 + Q^2 - 1}} + 1 \notag \\
&=\frac{1}{1-Q}.
    \label{tech-expr-2}
\end{align}
Returning to the original variables $(p,q)$, we have from (\ref{tech-expr-1}) and (\ref{tech-expr-2}) that 
$$
\frac{\partial T}{\partial p} \sim -\frac{1}{p-q}, \qquad 
\frac{\partial T}{\partial q} \sim \frac{1}{p-q},
$$
which yields the asymptotic expansion 
\begin{equation}
    \label{asympt-exp-temp}
T(p,q) = C_0 - \ln(p-q) + \tilde{T}(p,q),
\end{equation}
where $C_0$ is a uniquely specified constant and $\tilde{T}(p,q)$ is the remainder term to be estimated. 

To compute the constant $C_0$ in (\ref{asympt-exp-temp}), we recall the following solution of the second-order equation (\ref{ode-tilde}) for the homoclinic orbit with $E(p,q) = 0$:
$$
\tilde{u}(x) = \frac{3}{2} {\rm sech}^2\left(\frac{x+x_0}{2}\right),
$$
where $x_0 \in \mathbb{R}$ is arbitrary. As $x \to +\infty$, we get the expansion 
$$
\tilde{u}(x) = 6 e^{-(x+x_0)} + \mathcal{O}(e^{-2(x+x_0)}).
$$
Consider now the solution $\tilde{u}(x)$ on the interval $[0,L]$ from $(\tilde{u}(0),\tilde{u}'(0)) = (1,-\frac{1}{\sqrt{3}})$ 
to $(\tilde{u}(L),\tilde{u}'(L)) = (p_L,q_L)$. 
From $(\tilde{u}(0),\tilde{u}'(0)) = (1,-\frac{1}{\sqrt{3}})$, we obtain the unique value for $x_0$: 
$$
{\rm sech}^2\left(\frac{x_0}{2}\right) = \frac{2}{3} \quad \Rightarrow \quad 
x_0 = 2 \arccosh\left(\frac{\sqrt{3}}{\sqrt{2}}\right) > 0.
$$
From $(\tilde{u}(L),\tilde{u}'(L)) = (p_L,q_L)$, we obtain the expansion  
\begin{equation}
    \label{tech-expr-3}
p_L = 6 e^{-(L+x_0)} + \mathcal{O}(e^{-2(L+x_0)}), \quad 
q_L = -6 e^{-(L+x_0)} + \mathcal{O}(e^{-2(L+x_0)}), \quad \mbox{\rm as} \;\; L \to \infty.
\end{equation}
By using expansions (\ref{asympt-exp-temp}) and (\ref{tech-expr-3}) in $T(p_L,q_L) = L$, we obtain 
\begin{equation}
    \label{tech-expr-4}
L = C_0 - \ln(12) + L + x_0 + \mathcal{O}(e^{-L}),  \quad \mbox{\rm as} \;\; L \to \infty.
\end{equation}
It follows from (\ref{tech-expr-4}) that $C_0 = \ln(12) - x_0$ since the constant $C_0$ is independent of the parameter $L$. 
This yields the asymptotic expansion (\ref{asymptotics-T}), provided that $\tilde{T}(p,q) = \mathcal{O}(p)$ as $p \to 0^+$, $q = Qp$.

Finally, we justify the reminder term $\tilde{T}(p,q)$ in (\ref{asympt-exp-temp}). We rewrite (\ref{dTdp}) as 
\begin{align*}
    p \frac{\partial T}{\partial p} &= - \frac{p^2 (1-p)}{1 + 3(q^2-p^2) + 2p^3} \int_{p}^1 \frac{1-u^2}{u^2 \sqrt{q^2 - p^2 + u^2 + \frac{2}{3} (p^3-u^3)}} du + \frac{3pq}{1 + 3(q^2-p^2) + 2p^3} \\
    &= - p^2 [1 + \mathcal{O}(p)] \int_{p}^1 \frac{1-u^2}{u^2 \sqrt{(Q^2 - 1) p^2 + u^2 + \frac{2}{3} (p^3-u^3)}} du + \mathcal{O}(p^2), 
\end{align*}
as $p \to 0^+$ and $q = Q p \to 0^-$. We need to prove that 
\begin{align*}
- p^2 \int_{p}^1 \frac{1-u^2}{u^2 \sqrt{(Q^2 - 1) p^2 + u^2 + \frac{2}{3} (p^3-u^3)}} du = -\frac{1}{1-Q} + \mathcal{O}(p). 
\end{align*}
To show this, we write 
\begin{align*}
& \qquad - p^2 \int_{p}^1 \frac{1-u^2}{u^2 \sqrt{(Q^2 - 1) p^2 + u^2 + \frac{2}{3} (p^3-u^3)}} du \\
&= 
- p^2 \int_{p}^1 \frac{du}{u^2 \sqrt{(Q^2 - 1) p^2 + u^2 + \frac{2}{3} (p^3-u^3)}} + p^2 T(p,q),
\end{align*}
where $p^2 T(p,q) = \mathcal{O}(p^2 |\ln(p)|)$ is much smaller than $\mathcal{O}(p)$ due to the leading-order term in (\ref{asympt-exp-temp}). On the other hand, we use the same substitution $u = px$ and obtain 
\begin{align*}
- p^2 \int_{p}^1 \frac{du}{u^2 \sqrt{(Q^2 - 1) p^2 + u^2 + \frac{2}{3} (p^3-u^3)}} &= 
- \int_{1}^{1/p} \frac{dx}{x^2 \sqrt{x^2 + Q^2 - 1 + \frac{2}{3} p (1-x^3)}} \\
&= -\int_{1}^{1/p} \frac{dx}{x^2 \sqrt{x^2 + Q^2 - 1}} + {\rm Rem}, 
\end{align*}  
where 
\begin{align*}
{\rm Rem} &= \int_{1}^{1/p} \frac{dx}{x^2 \sqrt{x^2 + Q^2 - 1}} - \int_{1}^{1/p} \frac{dx}{x^2 \sqrt{x^2 + Q^2 - 1 + \frac{2}{3} p (1-x^3)}} \\
&= \frac{2}{3} p \int_1^{1/p} \frac{(1-x^3) dx}{x^2 A B (A+B)},
\end{align*} 
with $A := \sqrt{x^2 + Q^2 - 1}$ and $B := \sqrt{x^2 + Q^2 - 1 + \frac{2}{3} p (1-x^3)}$. By the same Lebesgue's Dominated Convergence Theorem, we have 
\begin{align*}
\lim_{p \to 0^+} \int_1^{1/p} \frac{(1-x^3) dx}{x^2 A B (A+B)} = \int_1^{\infty} \frac{(1-x^3) dx}{2 x^2 \sqrt{[x^2 + Q^2 - 1]^3}} < \infty,
\end{align*} 
which proves that ${\rm Rem} = \mathcal{O}(p)$ as $p \to 0^+$. Finally, we have 
\begin{align*}
-\int_{1}^{1/p} \frac{dx}{x^2 \sqrt{x^2 + Q^2 - 1}} = -\int_{1}^{\infty} \frac{dx}{x^2 \sqrt{x^2 + Q^2 - 1}} + 
\int_{1/p}^{\infty} \frac{dx}{x^2 \sqrt{x^2 + Q^2 - 1}},
\end{align*}  
where the first term is $-\frac{1}{1-Q}$ according to the explicit computation above and the second term is of the order of $\mathcal{O}(p^2)$. Thus, we have proven that 
$$
p \frac{\partial T}{\partial p} = -\frac{1}{1-Q} + \mathcal{O}(p) \quad \mbox{\rm as} \quad p \to 0^+, \;\; q = Q p, 
$$
which suggests in view of (\ref{asympt-exp-temp}) that $\frac{\partial \tilde{T}}{\partial p} = \mathcal{O}(1)$ so that $\tilde{T}(p,q) = \mathcal{O}(p)$ as $p \to 0^+$. 
\end{proof}

\begin{remark}
	It follows from (\ref{der-p}) and (\ref{der-q}) that the period function $T(p,q)$ satisfies the first-order PDE
	$$
	q \frac{\partial T}{\partial p} + p(1-p) \frac{\partial T}{\partial q} = 1.
	$$
	The leading-order part of the asymptotic expansion (\ref{asympt-exp-temp}) satisfies the truncated equation 
	$$
	T^{(0)}(p,q) = C_0 - \ln(p-q) : \quad q \frac{\partial T^{(0)}}{\partial p} + p \frac{\partial T^{(0)}}{\partial q} = 1,
	$$
	whereas the remainder term $\tilde{T}(p,q)$ satisfies 
\begin{equation}
\label{pde-for-T}
q \frac{\partial \tilde{T}}{\partial p} + p(1-p) \frac{\partial \tilde{T}}{\partial q} = p^2 \frac{\partial T^{(0)}}{\partial q} = \frac{p^2}{p-q} = \mathcal{O}(p)
\end{equation}
	as $p \to 0^+$ and $q = Qp$. The estimates of Lemma \ref{lem-period-1-asymptotics} for $\tilde{T}(p,q) = \mathcal{O}(p)$ are compatible with (\ref{pde-for-T}). However, equation (\ref{pde-for-T}) is not useful for the analysis of the remainder terms because the homogeneous solutions of (\ref{pde-for-T}) may exhibit strong singularities as $p \to 0^+$ and $q = Qp$.
\end{remark}

\begin{remark}
The asymptotic behavior of the period function $T(p,q)$ is similar to (\ref{asymptotics-T}) but for the second-order equation with the cubic power was obtained in \cite{KP21} based on the explicit expansion of Jacobi elliptic functions with respect to the elliptic modulus near a homoclinic orbit, which was studied in \cite{Berkolaiko}. However, we have found for the second-order equation (\ref{ode-tilde}) with the quadratic power that the explicit expansion of Jacobi elliptic functions for the solution $\tilde{u} = \tilde{u}(x)$ yields vanishing first-order corrections with respect to the elliptic modulus. As a result, the asymptotic behavior of $T(p,q)$ is obtained in Lemma \ref{lem-period-1-asymptotics} by studying the weakly singular integrals directly.
\end{remark}

\subsection{The period function from the point $(p,q)$}

Let $(p,q) \in (0,1) \times (-\infty,0)$ be again a point on the phase plane of the 
second-order equation (\ref{ode-tilde}). We are now considering the integral curve between the points 
$(p,q)$ and $(p_0,0)$, where $p_0 = p_0(p,q) \in (0,1)$ is uniquely found from the implicit equation given by 
$$
E(p,q) = q^2 - p^2 + \frac{2}{3} p^3 = -p_0^2 + \frac{2}{3} p_0^3.
$$
The period function for this part of the integral curve between the two points is given by 
\begin{equation}
    \label{period-function-second}
T_0(p,q) = \int_{p_0}^p \frac{du}{v}, \quad v = \sqrt{E(p,q)+A(u)}, \quad A(u) = u^2 - \frac{2}{3} u^3.
\end{equation}
Figure \ref{fig:enter-label} shows the part of the integral curve between points $(p,q)$ and $(p_0,0)$.

\begin{figure}[htb!]
    \centering
    \includegraphics[width=8cm,height=5.5cm]{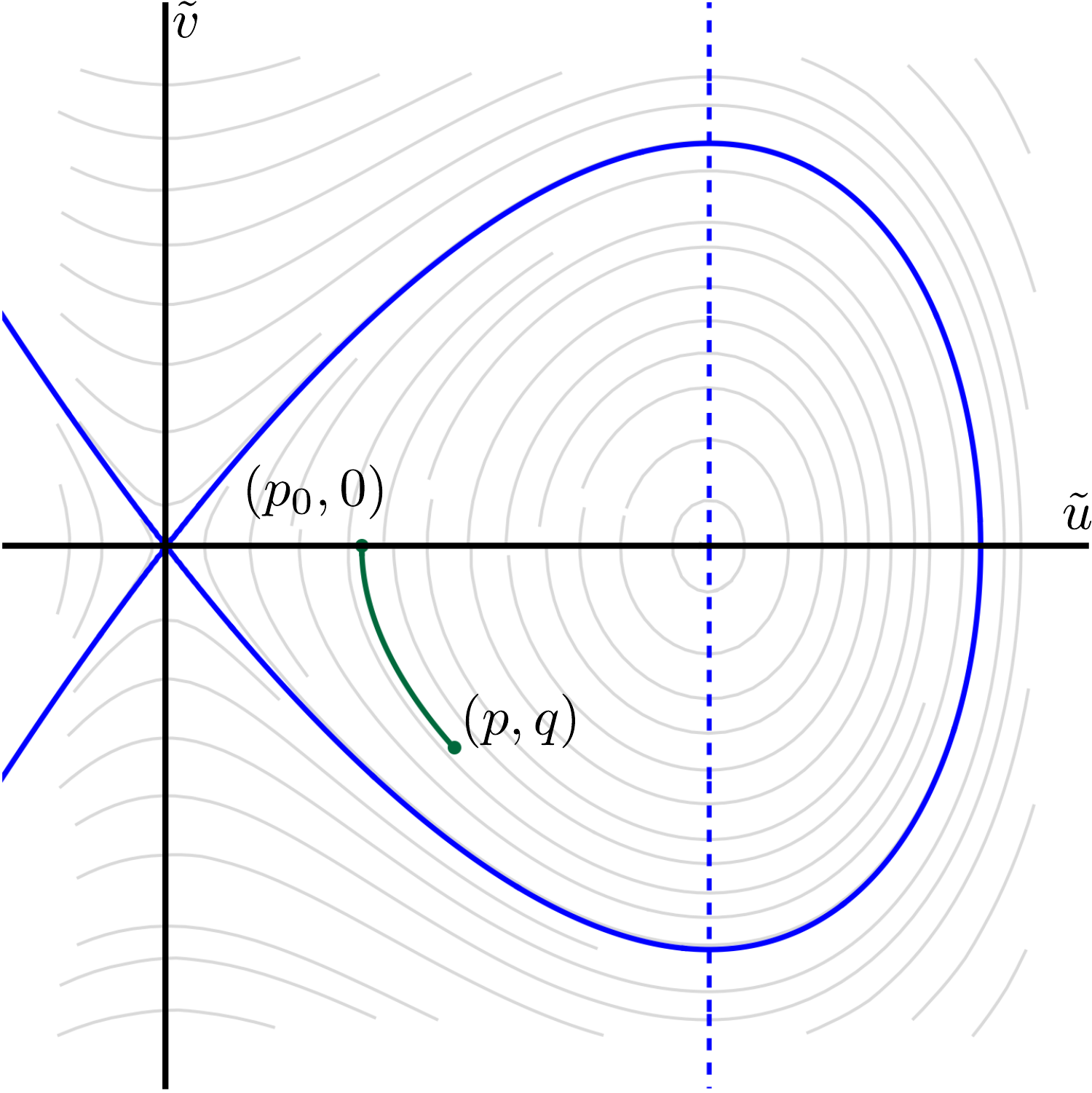}
    \caption{An example of the period function $T_0(p,q)$ given by (\ref{period-function-second}) in terms of the phase portrait obtained from the level curves of $E(\tilde{u},\tilde{v})$ given by (\ref{invariant}). The homoclinic orbit is in blue, the integral curve from $(p,q)$ to $(p_0,0)$ is in green, and the other orbits are in gray. }
    \label{fig:enter-label}
\end{figure}

The following lemma gives the monotonicity  of $T_0(p,q)$ with respect to two independent parameters $(p,q)$. 

\begin{lemma}
	\label{lem-period-2}
	For every $(p,q) \in (0,1) \times (-\infty,0)$, we have 
	\begin{equation}
	\label{monotonicity-second}
	\frac{\partial T_0}{\partial q} < 0.
	\end{equation} 
Moreover, if $(p,q) \in (0,\frac{1}{2}] \times (-\infty,0)$, then 
	\begin{equation}
	\label{monotonicity-third}
	\frac{\partial T_0}{\partial p} < 0.
    \end{equation}
\end{lemma} 

\begin{proof}
    Similarly to the proof of Lemma \ref{lem-period-1}, we write 
   	\begin{align*}
[E(p,q) + A(1)] T_0(p,q) = \int_{p_0}^{p} \left(3 - \frac{2 A''(u) [A(u) - A(1)]}{[A'(u)]^2} \right) v du -\frac{(1-p)(1+2p)}{3p} q.
\end{align*}
The right-hand side is continuously differentiable in $(p,q)$ from which we obtain 
	\begin{align}
    \label{dT0dp}
[E(p,q) + A(1)] \frac{\partial T_0}{\partial p} &= -p(1-p) \int_{p_0}^p \frac{1-u^2}{3u^2 v} du - q
\end{align}
and 
	\begin{align}
        \label{dT0dq}
[E(p,q) + A(1)] \frac{\partial T_0}{\partial q} &= q \int_{p_0}^{p} \frac{1-u^2}{3u^2 v} du - \frac{(1-p)(1+2p)}{3p}.
\end{align}
Since $E(p,q) + A(1) > 0$ for $(p,q) \in (0,1) \times (-\infty,0)$, we have $\frac{\partial T_0}{\partial q} < 0$, which yields 
(\ref{monotonicity-second}). 

To get monotonicity of $T_0$ with respect to $p$, we also obtain 
   	\begin{align*}
E(p,q) T_0(p,q) = \int_{p_0}^{p} \left(3 - \frac{2 A''(u) A(u)}{[A'(u)]^2} \right) v du + \frac{p (3-2p)}{3 (1-p)} q,
\end{align*}
which is allowed for the integral on $[p_0,p]$ with $0 < p_0 < p < 1$. Differentiating in $p$ yields now 
\begin{align}
E(p,q) \frac{\partial T_0}{\partial p} &= -p(1-p) \int_{p_0}^p \left[1 - \frac{2 A''(u) A(u)}{[A'(u)]^2} \right] \frac{du}{v} \notag \\
& \quad -q \left[3 - \frac{2 A''(p) A(p)}{[A'(p)]^2} \right] +  q \frac{d}{dp} \frac{p (3-2p)}{3 (1-p)} \notag \\
&= -p(1-p) \int_{p_0}^p \frac{u(2-u)}{3 (1-u)^2 v} du - q. \label{tech-expr-5}
\end{align}
The right-hand sides of both (\ref{dT0dp}) and (\ref{tech-expr-5}) contain the sum of a negative and a positive term. By subtracting one expression from another, the positive term $-q$ is canceled out and we obtain 
\begin{align*}
\frac{\partial T_0}{\partial p} &= -p(1-p) \int_{p_0}^p \frac{(1-2u)}{u^2 (1-u)^2 v} du,
\end{align*}
where we have used that $A(1) = \frac{1}{3}$. The right-hand side is strictly negative if $p \in (0,\frac{1}{2}]$ for which $1 - 2 u \geq 0$ for $u \in [p_0,p]$. This yields (\ref{monotonicity-third}). 
\end{proof}

The following lemma gives the singular behavior of the two period functions $T(p,q)$ and $T_0(p,q)$ as $(p,q) \to (1,0)$. It is used in the proof of 
Proposition \ref{lem-lower-bound} below. 

\begin{lemma}
	\label{lem-period-3}
Let $T(p,q)$ and $T_0(p,q)$ be defined by (\ref{period-function}) and (\ref{period-function-second}), respectively. Then, $T(p,q)$ satisfies the asymptotic expansion
	\begin{align}
\label{center-point-asymptotics}
T(p,q) = \arcsin \frac{1-p}{\sqrt{(1-p)^2+q^2}} + \mathcal{O}(|1-p|), \quad 
\mbox{\rm as} \;\; p \to 1^-, \;\; q = Q(1-p), 
\end{align} 
whereas $T_0(p,q)$ has the limit 
\begin{align}
\label{center-point-limit}
\lim_{q = Q(1-p), p \to 1^-} T_0(p,q) = \frac{\pi}{2} - \arcsin \frac{1-p}{\sqrt{(1-p)^2+q^2}},
\end{align} 
where $Q \in (-\infty,0)$ is fixed.
\end{lemma} 

\begin{proof}
We obtain from (\ref{energy-level}) and (\ref{period-function}) that 
\begin{align*}
    T(p,q) &= \int_p^1 \frac{du}{\sqrt{q^2 - p^2 + \frac{2}{3} p^3 + u^2 - \frac{2}{3} u^3}} \\
    &= \int_{p-1}^0 \frac{dx}{\sqrt{(1-p)^2 + q^2 - x^2 - \frac{2}{3} (1-p)^3 - \frac{2}{3} x^3}} \\
    &= \int_{\arcsin\frac{p-1}{\sqrt{(1-p)^2 + q^2}}}^0 \frac{\cos(t) dt}{\sqrt{\cos^2(t) - \frac{2}{3} \frac{(1-p)^3}{(1-p)^2 + q^2} - \frac{2}{3} \sqrt{(1-p)^2 + q^2}  \sin(t)^3}} \\
     &= \int_{-\arcsin\frac{1}{\sqrt{1 + Q^2}}}^0 \frac{\cos(t) dt}{\sqrt{\cos^2(t) - \frac{2}{3} \frac{(1-p)}{1 + Q^2} - \frac{2}{3} \sqrt{1 + Q^2} (1-p) \sin(t)^3}},
\end{align*}
where we have used substitutions $u = 1 + x$ and $x = \sqrt{q^2 + (1-p)^2} \sin(t)$ as well as $q = Q(1-p)$ with fixed $Q \in (-\infty,0)$. The limit $p \to 1^-$ is now well defined due to Lebesgue's Dominated Convergence Theorem and yields 
$$
\lim_{q = Q(1-p), p \to 1^-} T(p,q) = \arcsin \frac{1}{\sqrt{1+Q^2}},
$$
which gives the leading-order term in (\ref{center-point-asymptotics}). Similarly, we have 
\begin{align*}
    T_0(p,q) &= \int_{p_0}^1 \frac{du}{\sqrt{q^2 - p^2 + \frac{2}{3} p^3 + u^2 - \frac{2}{3} u^3}} \\
     &= \int_{-\arcsin\frac{1-p_0}{(1-p) \sqrt{1 + Q^2}}}^{-\arcsin\frac{1}{\sqrt{1 + Q^2}}} \frac{\cos(t) dt}{\sqrt{\cos^2(t) - \frac{2}{3} \frac{(1-p)}{1 + Q^2} - \frac{2}{3} \sqrt{1 + Q^2} (1-p) \sin(t)^3}},
\end{align*}
where we have used the definition 
$$
(1 - p_0)^2 - \frac{2}{3} (1-p_0)^3 = (1-p)^2 - \frac{2}{3} (1-p)^3 + q^2
$$
Since 
$$
\lim_{q = Q(1-p), p \to 1^-}  \frac{1-p_0}{1-p} = \sqrt{1+Q^2}, 
$$
we obtain 
$$
\lim_{q = Q(1-p), p \to 1^-} T_0(p,q) = \frac{\pi}{2} - \arcsin \frac{1}{\sqrt{1+Q^2}}.
$$
which gives the limit in (\ref{center-point-limit}). 

In order to justify the reminder term in (\ref{center-point-asymptotics}), we write
\begin{align*}
    {\rm Rem} &:= T(p,q) - \arcsin \frac{1-p}{\sqrt{(1-p)^2+q^2}} \\
    &= \frac{2 (1-p)}{3 (1+Q^2)} \int_{-\arcsin\frac{1}{\sqrt{1 + Q^2}}}^0 \frac{1 + (1+Q^2)^{3/2} \sin^3(t)}{A [ \cos(t) + A] } dt,
\end{align*}
where $A := \sqrt{\cos^2(t) - \frac{2}{3} \frac{(1-p)}{1 + Q^2} - \frac{2}{3} \sqrt{1 + Q^2} (1-p) \sin(t)^3}$. The integral converges for every fixed $Q \in (-\infty,0)$ with the limit justified by Lebesgue's Dominated Convergence Theorem:
$$
\lim_{p \to 1^-} \int_{-\arcsin\frac{1}{\sqrt{1 + Q^2}}}^0 \frac{1 + (1+Q^2)^{3/2} \sin^3(t)}{A [ \cos(t) + A] } dt = 
\int_{-\arcsin\frac{1}{\sqrt{1 + Q^2}}}^0 \frac{1 + (1+Q^2)^{3/2} \sin^3(t)}{2 \cos^2(t)} dt < \infty.
$$
Hence, the remainder term in (\ref{center-point-asymptotics}) is of the order of $\mathcal{O}(|1-p|)$. 
\end{proof}

\begin{remark}
    The justification of the remainder term developed for $T(p,q)$ in the proof of Lemma \ref{lem-period-3} does not work for $T_0(p,q)$ since the lower limit of integration converges to $\frac{\pi}{2}$ as $p \to 1^-$ and the integral of $1/\cos^2(t)$ diverges at the lower limit. Hence, Lebesgue's Dominated Convergence Theorem cannot be used and the order of the remainder term is not obtained. However, we are only using the limit (\ref{center-point-limit}) in the proof of Proposition \ref{lem-lower-bound}. 
\end{remark}

\subsection{The unique positive state in the limit of long graphs}

Based on Lemma \ref{lem-period-1-asymptotics} and the linear estimates from \cite{Berkolaiko}, we obtain another proof of the uniqueness of the ground state as well as  the exponential smallness of the ground state in the limit of long graphs. 

\begin{theorem}
	\label{theorem-localized-state}
There exist $L_* > L_0$ such that if $L_{\rm min} := \min_j L_j > L_*$, then the positive ground state $u_*$ of $H(u)$ is unique and satisfies 
\begin{equation}
\label{bound-on-u}
\| u_* - 1\|_{L^{\infty}(\Gamma_0)} \leq C e^{-L_{\rm min}},
\end{equation}
where $C$ is a positive constant and $\Gamma_0$ is the part of $\Gamma$ without the pendants.
\end{theorem}

\begin{proof}
Consider the subset $V_0$ of interior vertices $V = \{ v_j \}$ of the graph $\Gamma$ which are also the boundary vertices between the set of pendants denoted by $P_0$ and $\Gamma_0 = \Gamma \backslash P_0$. We add the inhomogeneous Dirichlet conditions $p_j = \tilde{u}(v_j) \geq 0$ at $v_j \in V_0 \subset V$, where $\tilde{u}(x) = 1 - u(x)$. The proof is obtained by partioning the graph $\Gamma$ into $\Gamma_0$ and $P_0$ with two elliptic equations: 
\begin{equation}
	\label{bvp}
	\left\{ \begin{array}{l} -\Delta_{\Gamma_0} \tilde{u} + \tilde{u} = \tilde{u}^2 \;\; \mbox{\rm in} \; \Gamma_0, \\
	\tilde{u} \;\;\mbox{\rm satisfies NK conditions on} \;\; V \backslash V_0, \\
	\tilde{u}(v_j) = p_j \geq 0, \quad v_j \in V_0,
	\end{array} \right.
\end{equation} 
	and 
\begin{equation}
\label{segment}
\left\{ \begin{array}{l} -\tilde{u}''(x) + \tilde{u}(x) = \tilde{u}(x)^2 \;\; \mbox{\rm in} \; P_0, \\
\tilde{u} \;\;\mbox{\rm satisfies Dirichlet condition on the boundary vertices}, \\
\tilde{u}(v_j) = p_j \geq 0, \quad v_j \in V_0.
\end{array} \right.
\end{equation}	
In what follows, we obtain the unique solutions on $\Gamma_0$ and on $P_0$ for any given $\vec{p} := \{p_j\}_{v_j \in V_0}$ 
defined in a ball of small radius $p_0$ such that $\| \vec{p} \| \leq p_0$. The elliptic equation (\ref{EL}) on the graph $\Gamma$ 
is satisfied if $\vec{p}$ is uniquely found from the NK conditions on $V_0$. 

By Theorem 2.9 in \cite{Berkolaiko} (modified from the cubic to quadratic nonlinearities), there exist $C_0 > 0$, $p_0 > 0$, $L_* > 0$ such that if $L_{\rm min} = \min_j L_j > L_*$, then for every $\vec{p}$ such that $\| \vec{p} \| \leq p_0$, there exists a unique solution $\tilde{u} \in D(\Delta_{\Gamma_0})$ of the boundary-value problem (\ref{bvp}) in the class of functions with $\tilde{u}(x) \in [0,1]$ for every $x \in \Gamma_0$. Moreover, the solution satisfies the estimates 
\begin{equation}
\label{bvp-1}
\| \tilde{u} \|_{L^{\infty}(\Gamma_0)} \leq C_0 \| \vec{p} \|
\end{equation}
and
\begin{equation}
\label{bvp-2}
|q_j - d_j p_j| \leq C_0 \left( \| \vec{p} \| e^{-L_{\rm min}} + \| \vec{p}\|^2 \right), \quad v_j \in V_0, 
\end{equation}
where $q_j$ is the Neumann data (the sum of outward derivatives from $\Gamma_0$) at the vertex $v_j$ and $d_j$ is the degree of the vertex $v_j \in V_0$. 

By Lemma \ref{lem-period-1}, there exists a unique solution $\tilde{u} \in C^{\infty}([0,L_j])$ of the boundary-value problem (\ref{segment}) 
for each $e_j \in P_0$ parameterized by $[0,L_j]$ with 
$$
(\tilde{u}(0),\tilde{u}'(0)) = (1,\tilde{q}_j) \quad {\rm and} \quad 
(\tilde{u}(L_j),\tilde{u}'(L_j)) = (p_j,q_j)
$$ 
in the class of functions monotonically decreasing on $[0,L_j]$, where $\tilde{q}_j$ is obtained from (\ref{energy-level}) and $q_j$ is obtained from $T(p_j,q_j) = L_j$, where $T(p,q)$ is defined in (\ref{period-function}). Moreover, it follows from (\ref{asymptotics-T}) that $q_j$ satisfies the expansion 
\begin{equation}
\label{segment-1}
q_j = p_j - 12 e^{-L_j-x_0} + \mathcal{O}(e^{-2L_j}),
\end{equation}
where $x_0 = 2 \arccosh\left(\frac{\sqrt{3}}{\sqrt{2}}\right)$ is not important in the limit of large $L_j$.

Assume that the vertex $v_j \in V_0$ has $m_j$ pendants in the set $P_0$ connected to it. It remains to satisfy the flux conditions at the vertices $v_j \in V_0$ from the sum of all outward derivatives from all $m_j$ pendants of $P_0$ and all $d_j$ edges of $V_0$ that connect to $v_j$. By using (\ref{bvp-2}) and (\ref{segment-1}), we obtain 
$$
d_j p_j + \mathcal{O}(\|\vec{p}\| e^{-L_{\rm min}} + \| \vec{p}\|^2) + 
m_j p_j - 12 \sum_{\ell_j \to v_j} e^{-L_j- x_0} + \mathcal{O}(e^{-2L_{\rm min}}) = 0,
$$ 
where the remainder terms are $C^1$ functions with respect to parameters $\vec{p}$ and the notation $\ell_j \to v_j$ only includes pendants of $P_0$ connected to the vertex $v_j$. By the implicit function theorem, there exists a unique solution in the small ball $\| \vec{p} \| \leq p_0$ for sufficiently small $p_0 > 0$ such that 
\begin{equation}
    \label{p-j-value}
p_j = \frac{12}{d_j + m_j} \sum_{\ell_j \to v_j} e^{-L_j- x_0} + \mathcal{O}(e^{-2L_{\rm min}}). 
\end{equation}
This proves the uniqueness of solutions in the class of functions with $u(x) \in [0,1]$ for $x \in \Gamma$. The bound (\ref{bound-on-u}) follows from 
the bound (\ref{bvp-1}). 
\end{proof}

\begin{remark}
    Depending on the value of $p_j$ in (\ref{p-j-value}), we can find $q_j$ from (\ref{segment-1}) either inside the homoclinic orbit if $q_j \in (-p_j,0)$ or outside the homoclinic orbit if $q_j \in (-\infty,-p_j)$.
\end{remark}

\section{The positive ground state in flower graphs}
\label{sec-5}

We construct the unique positive ground state in the particular case of a flower graph. We also illustrate the utility of the two period functions 
introduced in Section \ref{sec-4} for proving uniqueness of the positive ground state for $\lambda_0(\Gamma) \in (0,1)$.

A general flower graph is obtained for loops of different lengths. It generalizes particular examples of 
the interval graph, the tadpole graph, and the flower graph with loops of equal length, see Figure \ref{fig:flowers}. 
For transparency of computations, we consider examples of flower graphs in increasing order.

\subsection{The interval $[0,L]$} 
\label{sec-5-1}

We assume the Dirichlet condition at $0$ and the Neumann condition at $L$. Positive solutions can be parameterized by $p \in (0,1)$ 
in the period function $T(p,q)$ defined in (\ref{period-function}) with $q = 0$. Based on the monotonicity of $T(p,0)$ with respect to $p$ in Lemma \ref{lem-period-1}, we immediately get the following result. 

\begin{theorem}
	\label{theorem-example-1}
	Let $\Gamma = [0,L]$ with $u(0) = 0$ and $u'(L) = 0$. Then, the positive ground state exists for every $L \in (\frac{\pi}{2},\infty)$ and is unique. 
\end{theorem}

\begin{proof}
	For $q = 0$, we have explicitly 
	\begin{align*}
\frac{\partial T}{\partial p} &= -\frac{p}{(1-p)(1+2p)} \int_{p}^1 \frac{1-u^2}{u^2 v} du.
\end{align*}
Hence the mapping $(0,1) \ni p \to T(p,0) \in \mathbb{R}$ is strictly monotonically decreasing. Furthermore, $\lim_{p \to 0} T(p,0) = \infty$ since the integral curve is a part of the homoclinic orbit as $E(p,0) \to 0^-$ and 
\begin{align*}
\lim_{p \to 1} T(p,0) &= \lim_{p \to 1} \int_p^1 \frac{du}{\sqrt{u^2 - p^2 + \frac{2}{3} (p^3-u^3)}} \\
&=\lim_{\tilde{p} \to 0} \int_0^{\tilde{p}} \frac{d\tilde{u}}{\sqrt{\tilde{p}^2 - \tilde{u}^2 + \frac{2}{3} (\tilde{u}^3-\tilde{p}^3)}} \\
&=\lim_{\tilde{p} \to 0} \int_0^1 \frac{dx}{\sqrt{1 - x^2 + \frac{2}{3} \tilde{p} (x^3-1)}} \\
&=\int_0^1 \frac{dx}{\sqrt{1-x^2}} = \frac{\pi}{2},
\end{align*}	
where we used the substitutions $\tilde{u} = 1 - u$, $\tilde{p} = 1 - p$, and $\tilde{u} = \tilde{p} x$, and Lebesgue's Dominated Convergence Theorem. 
Thus for every $L \in (\frac{\pi}{2},\infty)$, there is a unique $p \in (0,1)$ and a unique integral curve with $E(p,0) = E(u,v)$ for the positive ground state.
\end{proof}

\subsection{A symmetric flower graph} 
\label{sec-5-2}

A symmetric flower graph consists of the line segment $[0,L]$ with the Dirichlet condition at the boundary vertex at $x = 0$ 
and $N$ equal loops parameterized as $[-L_0,L_0]$ and connected at the interior vertex at $x = L$ with the NK conditions:
\begin{align}
\label{BC-flower}
\left\{ \begin{array}{l}
    \tilde{u}(L) = \tilde{u}_j(-L_0) = \tilde{u}_j(L_0), \quad 1\leq j \leq N, \\
    \tilde{u}'(L) + \sum_{j=1}^N \tilde{u}_j'(L_0) - \tilde{u}_j'(-L_0) = 0,
    \end{array} \right.
\end{align}
where $\tilde{u}(x) : [0,L] \to \mathbb{R}$ and $\tilde{u}_j(x) : [-L_0,L_0] \to \mathbb{R}$, $1 \leq j \leq N$ are the components on the line segment and on the $N$ equal loops, respectively. Each component satisfies the second-order equation (\ref{ode-tilde}) on each edge of the flower graph. 
%If $N = 1$, we refer to this graph as the tadpole graph.

Let $(p,q) \in (0,1) \times (-\infty,0)$ be parameters of the point on the phase plane for the boundary conditions $p = \tilde{u}(L)$ and $q = \tilde{u}'(L)$ defined from $T(p,q) = L$, where $T(p,q)$ is the first period function introduced in (\ref{period-function}). The strictly positive ground state is necessarily described by the identical and even functions $\tilde{u}_j = \tilde{u}_0$, where $\tilde{u}_0(x) : [-L_0,L_0] \to \mathbb{R}$ satisfies the boundary conditions $\tilde{u}_0(\pm L_0) = p$ and $\tilde{u}'_0(\pm L_0) = \mp \frac{q}{2N}$ due to the NK conditions (\ref{BC-flower}). The second period function introduced in (\ref{period-function-second}) gives the second condition $T_0(p,\frac{q}{2N}) = L_0$ for the existence of the strictly positive ground state on the flower graph. Combining together, the existence of the strictly positive ground state is equivalent to finding a root $(p,q) \in (0,1) \times (-\infty,0)$ of the system of two equations: 
\begin{equation}
    \label{period-equation}
T(p,q) = L, \qquad T_0\left(p,\frac{q}{2N} \right) = L_0. 
\end{equation}

The following theorem employs monotonicity of the two period functions (\ref{period-function}) and (\ref{period-function-second}) in Lemmas 
\ref{lem-period-1} and \ref{lem-period-2}, respectively, in order to prove uniqueness of the positive ground state.

\begin{theorem}
	\label{theorem-example-2}
	Let $\Gamma = [0,L] \times \underbrace{[-L_0,L_0] \times \dots \times [-L_0,L_0]}_{N \; \mathrm{times}}$ be the symmetric flower graph with the Dirichlet condition at the boundary vertex and the NK conditions (\ref{BC-flower}) at the interior vertex. Then, there exists a simply connected region $\Omega \in \mathbb{R}^+ \times \mathbb{R}^+$ such that the positive ground state exists for every $(L,L_0) \in \Omega$ and is unique. 
\end{theorem}

\begin{proof}
We need to compute a solution $(p,q) \in (0,1) \times (-\infty,0)$ of the two nonlinear equations 
(\ref{period-equation}) for a given point $(L,L_0) \in \mathbb{R}^+ \times \mathbb{R}^+$. To show 
invertibility of the transformation $(p,q) \to (L,L_0)$, we compute its Jacobian from the derivatives of the period functions 
$T(p,q)$ and $T_0(p,\frac{q}{2N})$ with respect to $(p,q)$. It follows from  
(\ref{dTdp}), (\ref{dTdq}), (\ref{dT0dp}), and (\ref{dT0dq}) that 
\begin{align*}
[E(p,q) + A(1)] \frac{\partial T}{\partial p} &= -p(1-p) \mathcal{I}_1 + q, \\
[E(p,q) + A(1)] \frac{\partial T}{\partial q} &= q \mathcal{I}_1 + \frac{(1-p)(1+2p)}{3p}, \\
[E\left( p,\frac{q}{2N} \right) + A(1)] \frac{\partial T_0}{\partial p} &= -p(1-p) \mathcal{I}_2 - \frac{q}{2N}, \\
[E\left( p,\frac{q}{2N} \right) + A(1)] \frac{\partial T_0}{\partial q} &= \frac{q}{4 N^2} \mathcal{I}_2 - \frac{(1-p)(1+2p)}{6N p}.
\end{align*}
where 
\begin{equation}
    \label{integrals}
\mathcal{I}_1 := \int_{p}^1 \frac{1-u^2}{3u^2 \sqrt{E(p,q) + A(u)}} du, \quad 
\mathcal{I}_2 := \int_{p_0}^p \frac{1-u^2}{3u^2 \sqrt{E(p,\frac{q}{2N}) + A(u)}} du. 
\end{equation}
We obtain from here that 
\begin{align*}
 \Delta(p,q) &:= [E(p,q) + A(1)]  [E\left( p,\frac{q}{2N} \right) + A(1)] \left(  \frac{\partial T}{\partial p} \frac{\partial T_0}{\partial q}  - \frac{\partial T}{\partial q}  \frac{\partial T_0}{\partial p} \right) \\
 &= \left( 1 - \frac{1}{4N^2} \right) q p (1-p) \mathcal{I}_1 \mathcal{I}_2 + \frac{1}{2N} \mathcal{I}_1 [E(p,q) + A(1)] 
 + \mathcal{I}_2 [E\left(p,\frac{q}{2N} \right) + A(1)] \\
 &= \left( 1 - \frac{1}{4N^2} \right)  p(1-p) \left[ q \mathcal{I}_1 + \frac{(1-p) (1+2p)}{3p} \right] \mathcal{I}_2 
 + \frac{1}{4N^2} (2N \mathcal{I}_1 + \mathcal{I}_2) [E(p,q) + A(1)],
\end{align*}
where we have used that 
\begin{align*}
E\left(p,\frac{q}{2N} \right) + A(1) - \left(1 - \frac{1}{4N^2} \right) \frac{(1-p)^2 (1+2p)}{3} = \frac{1}{4N^2} [E(p,q) + A(1)].
\end{align*}
Every term in the final expression for $\Delta(p,q)$ is positive 
since $\frac{\partial T}{\partial q} > 0$ in (\ref{dTdq}) 
and $E(p,q) + A(1) > 0$. Thus, the Jacobian is positive for every $(p,q) \in (0,1) \times (-\infty,0)$ so 
that, by the inverse function theorem, there exists a unique solution for $(p,q) \in (0,1) \times (-\infty,0)$ 
for every $(L,L_0)$ from a suitable existence region $\Omega \in \mathbb{R}^+ \times \mathbb{R}^+$. The region $\Omega$ is a simply connected range of the transformation $(p,q) \to (L,L_0)$.
\end{proof}

\begin{remark}
By Theorem \ref{theorem-localized-state}, there exists a unique positive ground state for any given (sufficiently large) $L$ and $L_0$. Hence, the existence region $\Omega$ is bounded from below near $(0,0)$ but extends along every ray for $(L,L_0)$ in $\mathbb{R}^+ \times \mathbb{R}^+$. Moreover, it follows from (\ref{segment}) and (\ref{p-j-value}) that 
$$
p = \frac{12}{1+2N} e^{-L-x_0} + \mathcal{O}(e^{-2L_{\rm min}}) \quad \mbox{\rm and} \quad 
q = -\frac{24 N}{1+2N} e^{-L-x_0} + \mathcal{O}(e^{-2L_{\rm min}}),
$$
in the limit of large $L_{\min} = \min(L,L_0)$. Since $q \in (-\infty,-p)$, the part of the solution $\tilde{u}(x) : [0,L] \to \mathbb{R}$ in the pendant with the Dirichlet condition is defined outside the homoclinic orbit on the phase plane for every $N \geq 1$ if $L_{\rm min}$ is sufficiently large.
\label{rem-outside}
\end{remark}

\begin{remark}
    If either $L \to 0$ or $L_0 \to 0$, the existence of the positive ground state on the flower graph reduces to the existence problem on the segments $[0,L]$ or $[0,L_0]$ with one Dirichlet and one Neumann boundary condition. Indeed, if $L_0 = 0$, then $q = 0$ and the system (\ref{period-equation}) reduces to the only equation $T(p,0) = L$ and if $L = 0$, then $p = 1$ and the system (\ref{period-equation}) reduces to the only equation 
    $$
    T_0\left(1,\frac{q}{2N}\right) = T(p_0(q),0) = L_0, 
    $$
where $p_0 = p_0(q) \in (0,1)$ is uniquely obtained from solutions of the cubic equation 
$$
\frac{q^2}{4N^2} - \frac{1}{3} = -p_0^2 + \frac{2}{3} p_0^3.
$$   
By Theorem \ref{theorem-example-1}, the existence region is bounded for $L > \frac{\pi}{2}$ if $L_0 = 0$ and for $L_0 > \frac{\pi}{2}$ if $L = 0$. The lower boundary of $\Omega$ is a curve in the $(L,L_0)$ plane which connects the points $(\frac{\pi}{2},0)$ and $(0,\frac{\pi}{2})$ such that the positive ground state exists if the point $(L,L_0)$ is located in $\Omega$ above the curve.
    \label{rem-lower}
\end{remark}

To get the lower boundary of the existence region $\Omega$ described in Remark \ref{rem-lower}, we need to use the asymptotic expressions for 
$T(p,q)$ and $T_0(p,q)$ as $(p,q) \to (1,0)$ in Lemma \ref{lem-period-3}. Indeed, by Theorem \ref{theorem-example-2}, the mapping $(p,q) \to (L,L_0)$ is invertible, hence the boundaries of $\Omega$ in the $(L,L_0)$ plane correspond to the boundaries of the vertical semi-strip $(0,1) \times (-\infty,0)$ in the $(p,q)$ plane. By Remark \ref{rem-lower}, the top boundary at $p \in (0,1)$, $q = 0$ is mapped to the boundary of $\Omega$ for $L \in \left(\frac{\pi}{2},\infty \right)$, $L_0 = 0$ and the right boundary at $p = 1$, $q \in (-\infty,0)$ is mapped to the boundary of $\Omega$ for $L = 0$, $L_0 \in \left(\frac{\pi}{2}, \infty \right)$. The lower boundary of $\Omega$ corresponds to the singular behavior of the two period functions at $(p,q) = (1,0)$, that is, at the center point of the second-order equation (\ref{ode-tilde}). Figure \ref{fig:existence-region} illustrates the region $(0,1) \times (-\infty,0)$ in the $(p,q)$ plane (left panel) and 
the region $\Omega$ in the $(L_0,L)$ plane with the lower boundary of $\Omega$ for $N = 1$ and $N = 5$ (right panel).

\begin{figure}[htb!]
    \centering
    \includegraphics[width=6cm,height=4cm]{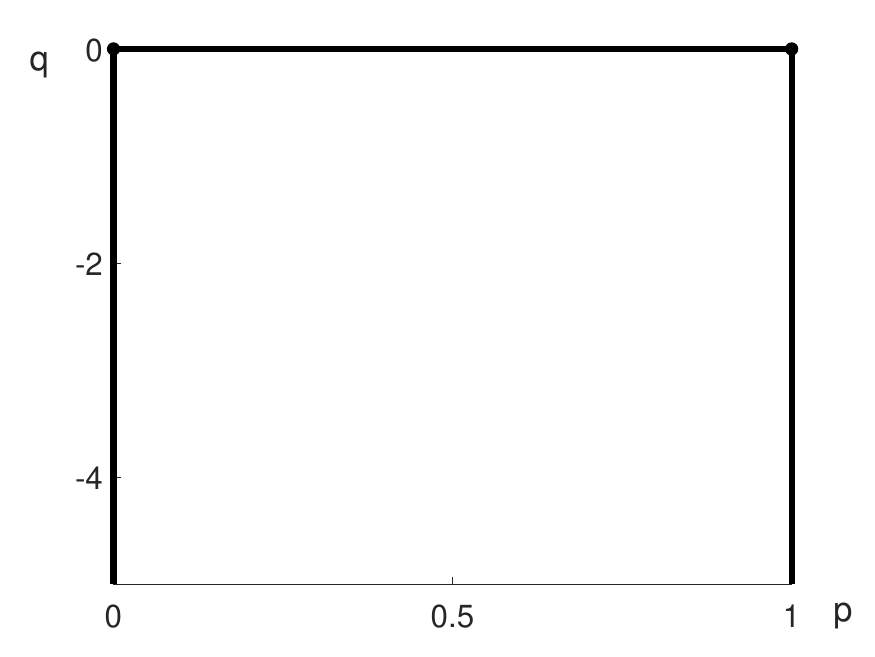}
    \includegraphics[width=6cm,height=4cm]{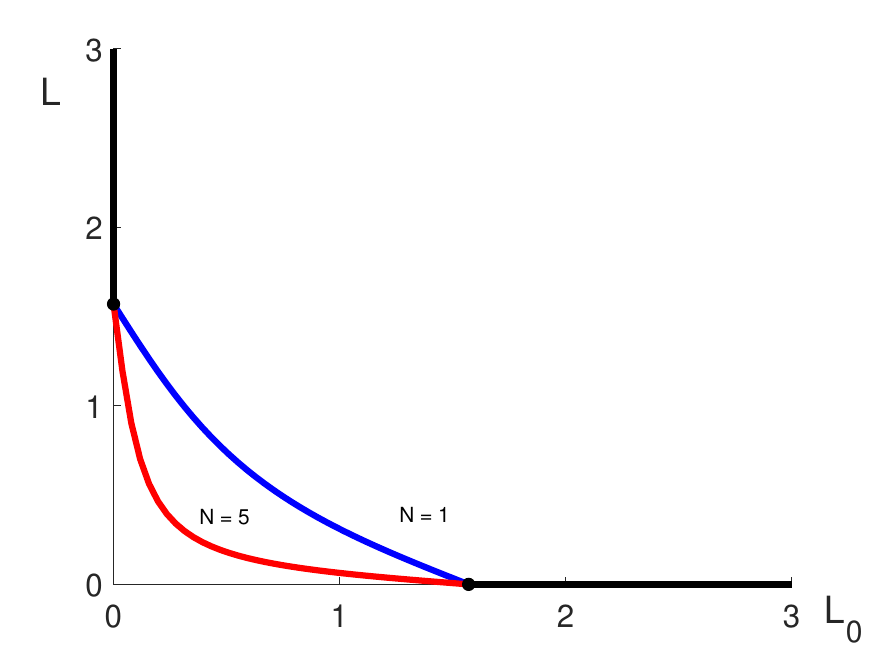}
    \caption{The existence region on the $(p,q)$ plane (left) and on the $(L_0,L)$ plane (right). The lower boundary of $\Omega$ is shown for  $N=1$ and $N=5$. As $N$ is increased, the curve approaches the axes.}
    \label{fig:existence-region}
\end{figure}

The following proposition gives the lower boundary of $\Omega$ based on Lemma \ref{lem-period-3}. 

\begin{proposition}
    \label{lem-lower-bound}
    The lower boundary of the existence region $\Omega \subset \mathbb{R}^+ \times \mathbb{R}^+$ is given by 
    \begin{equation}
        \label{lower-bound}
        L_0 = \arctan \frac{\cot(L)}{2N},
    \end{equation}
and it connects $L_0 = 0$ for $L = \frac{\pi}{2}$ and $L_0 = \frac{\pi}{2}$ if $L = 0$.
\end{proposition}

\begin{proof}
    We can solve the first equation of the system (\ref{period-equation}) for $q = q_L(p)$ with fixed $L > 0$ 
    thanks to the monotonicity of the period function $T(p,q)$ in Lemma \ref{lem-period-1}, where $q_L(p) \in (-\infty,0)$ is a unique solution of 
    $$
    T(p,q_L(p)) = L.
    $$
    Substituting the result into the second equation of the system (\ref{period-equation}) yields 
$$
T_0\left(p, \frac{q_L(p)}{2N} \right) = L_0,
$$
which can be uniquely solved for $(L,L_0) \in \Omega$ by Theorem \ref{theorem-example-2}. 
The lower bound of the existence region $\Omega$ is defined by 
\begin{equation}
\label{lower-bound-expression}
L_0 = \min_{p \in [0,1]} T_0\left(p, \frac{q_L(p)}{2N} \right) = \lim_{p \to 1^-} T_0\left(p, \frac{q_L(p)}{2N} \right),
\end{equation}
where we have used the argument that the lower bound corresponds to the corner point $(p,q) = (1,0)$ in the boundary of the vertical semi-strip $(0,1) \times (-\infty,0)$ in the $(p,q)$ plane mapped into the $(L,L_0)$-plane, see Figure \ref{fig:existence-region}. By the asymptotic expression (\ref{center-point-asymptotics}), we compute 
the function $q_L(p)$ as $p \to 1^-$ from 
$$
L = \arcsin \frac{1-p}{\sqrt{(1-p)^2+q_L(p)^2}} + \mathcal{O}(|1-p|).
$$
This yields 
$$
\frac{1-p}{\sqrt{(1-p)^2+q_L(p)^2}} = \sin\left( L + \mathcal{O}(|1-p|) \right) \; \Rightarrow \; 
q_L(p) = -\cot(L) (1-p) + \mathcal{O}((1-p)^2),
$$
where $L \in \left(0,\frac{\pi}{2}\right)$ due to restriction $q_L(p) = Q (1-p)$ with $Q \in (-\infty,0)$. By using the limit (\ref{center-point-limit}) in (\ref{lower-bound-expression}), we obtain 
\begin{align*}
L_0 &= \frac{\pi}{2} - \lim_{p \to 1^-} \arcsin \frac{1-p}{\sqrt{(1-p)^2+\frac{q_L(p)^2}{4N^2}}} \\
&= \frac{\pi}{2} - \arcsin \frac{1}{\sqrt{1+\frac{\cot(L)^2}{4N^2}}},
\end{align*}
which yields (\ref{lower-bound}) by using elementary trigonometric identities.
\end{proof}

\begin{remark}
\label{rem-boundary}
By Lemma \ref{lem-var-1}, the lower boundary of $\Omega$ can be found from the threshold condition $\lambda_0(\Gamma) = 1$, 
where $\lambda_0(\Gamma)$ is the lowest eigenvalue of $-\Delta_{\Gamma}$ in $L^2(\Gamma)$. 
We can show that this definition coincides with the expression (\ref{lower-bound}) in Proposition \ref{lem-lower-bound}. 
The solution of $-\Delta_{\Gamma} \Psi = \lambda_0(\Gamma) \Psi$ is given pointwisely as 
    $$
    \psi(x) = \sin(\sqrt{\lambda_0(\Gamma)} x), \quad x \in (0,L)
    $$
    to satisfy the Dirichlet condition at $x = 0$ and 
$$
\psi_0(x) = \frac{\sin(\sqrt{\lambda_0(\Gamma)} L)}{\cos(\sqrt{\lambda_0(\Gamma)} L_0)} \cos(\sqrt{\lambda_0(\Gamma)} x), \quad x \in (-L_0,L_0),
$$
to satisfy the symmetry condition in the loops and the first NK condition in (\ref{BC-flower}). The second NK condition in 
(\ref{BC-flower}) yields 
\begin{equation}
    \label{transc-eq}
\tan(\sqrt{\lambda_0(\Gamma)} L_0) = \frac{1}{2N} \cot(\sqrt{\lambda_0(\Gamma)} L),
\end{equation}
which admits a unique root for $\lambda_0(\Gamma)$ in $\left( 0,\min\{ \frac{\pi^2}{4 L_0^2},\frac{\pi^2}{4 L^2}\} \right)$. It is easy to show that $\lambda_0(\Gamma)$ is monotonically decreasing with either $L$ or $L_0$, in agreement with Lemma \ref{lem-linear-1}. Comparing  (\ref{lower-bound}) 
and (\ref{transc-eq}) implies that $\lambda_0(\Gamma) = 1$ at the lower boundary of $\Omega$. 
\end{remark}

\begin{remark}
It follows from (\ref{lower-bound}) that $L_0$ decreases as $N$ increases for fixed $L$ and $L$ decreases as $N$ increases for fixed $L_0$. Hence, the lower boundary of the existence region $\Omega$ approaches the axis in the $(L,L_0)$  plane as $N$ increases. Figure \ref{fig:existence-region} (right) illustrates this phenomenon for $N = 1$ and $N = 5$. This implies that the positive ground state exists in a large region $\Omega$ for the flower graph with more loops. 
\end{remark}

Figure \ref{fig:n1example} illustrates the construction of the positive ground state on the tadpole graph with $N = 1$ by using parts of two integral 
curves of the second-order equation (\ref{ode-tilde}). 

\begin{figure}[htb!]
    \centering
    \includegraphics[width=5cm,height=4cm]{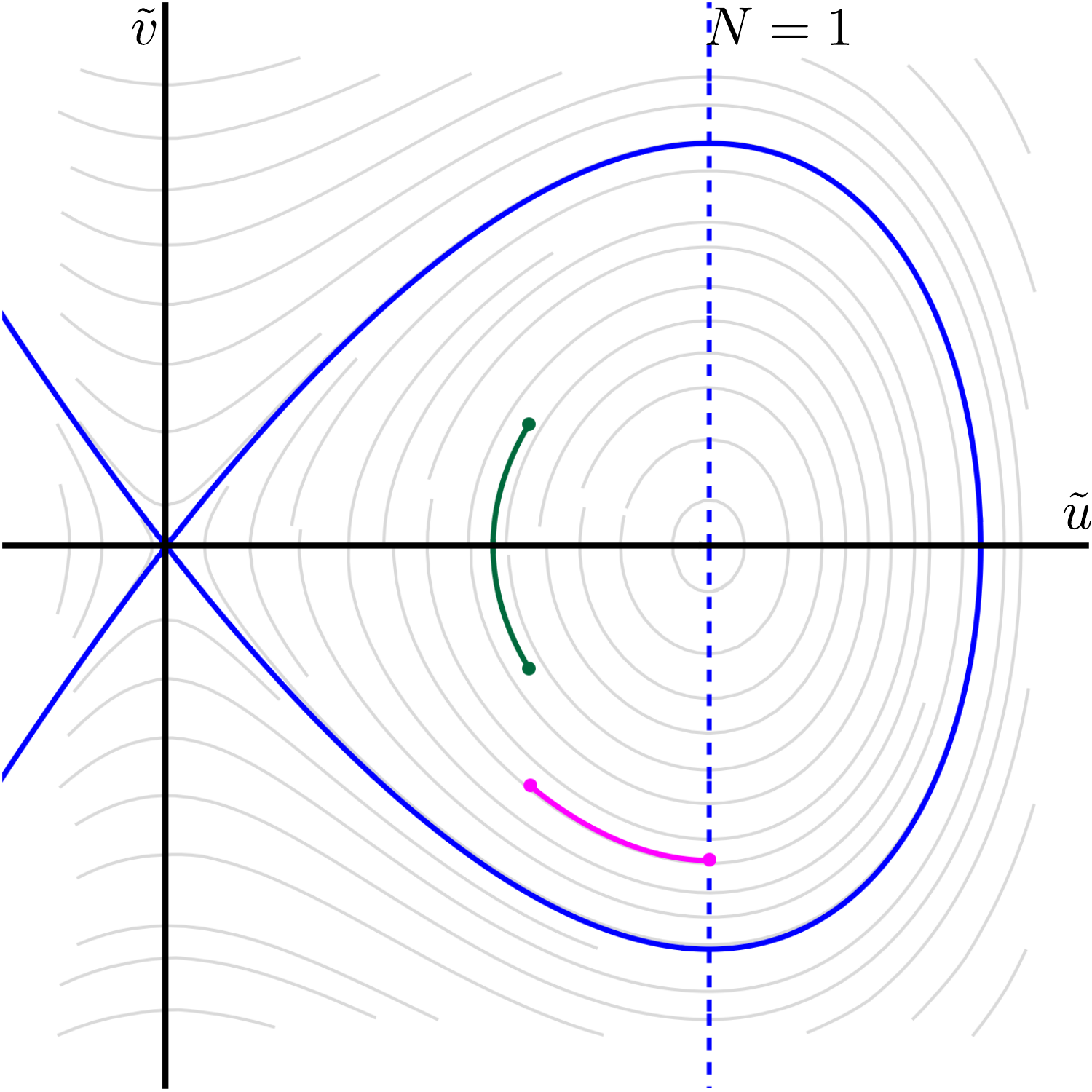}
    \includegraphics[width=5cm,height=4cm]{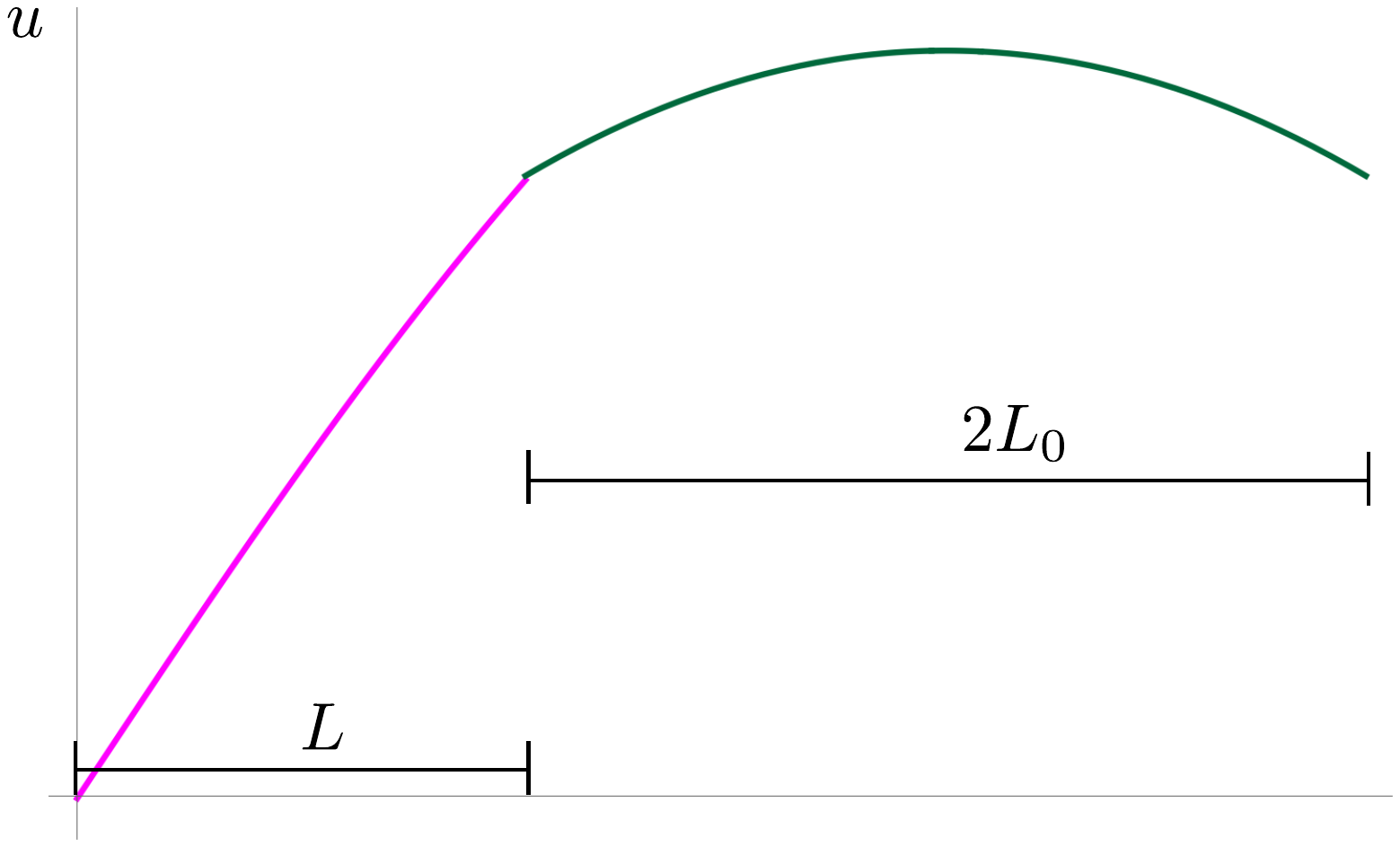}
    \includegraphics[width=5cm,height=4cm]{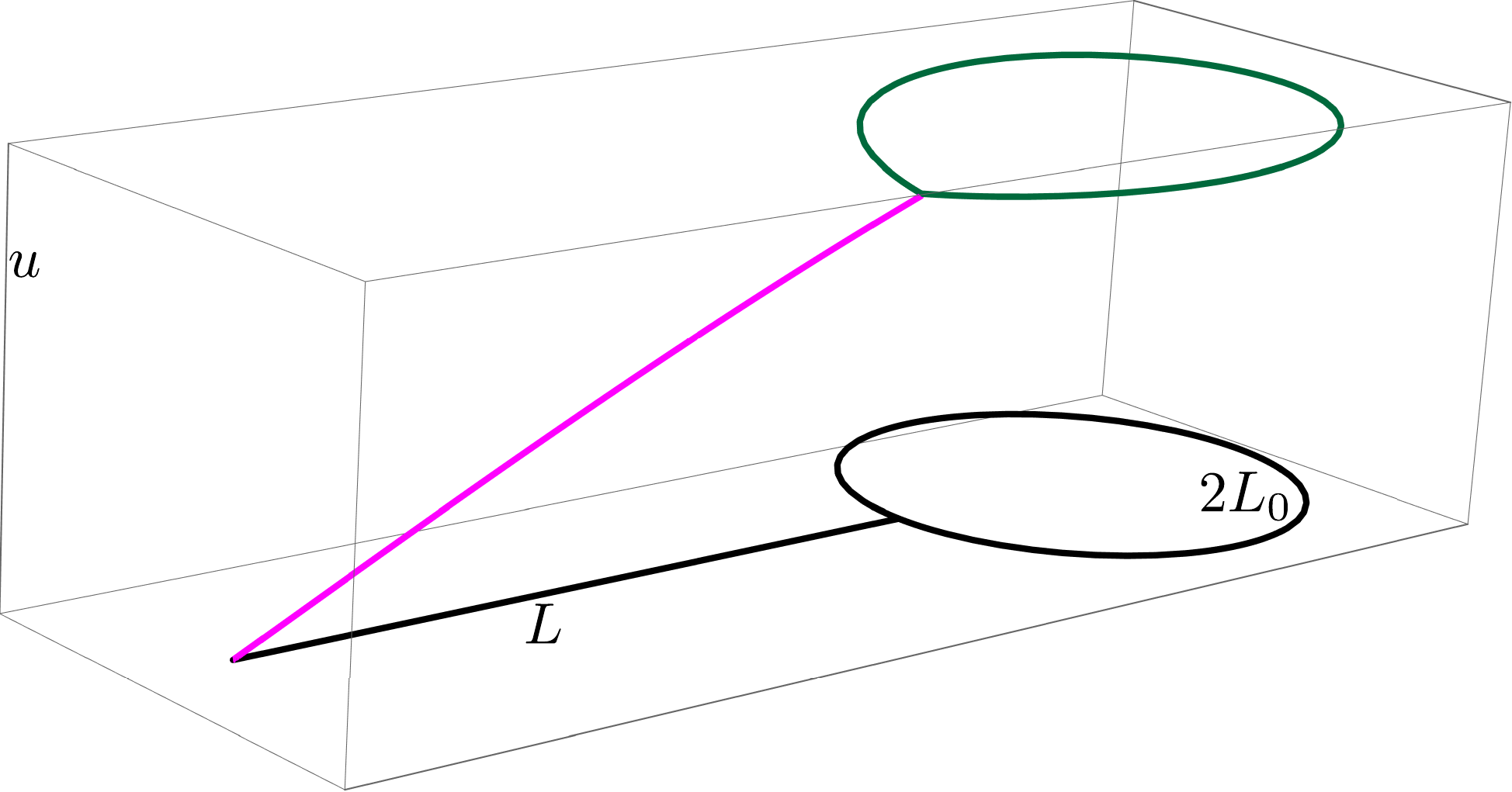}
    \caption{The positive ground state on a tadpole graph with a stem of length $L = 0.8$ and a loop of half-length $L_0 = 0.75$. Left: parts of two integral curves on the phase plane $(\tilde{u},\tilde{v})$. Center: a plot in variables $(x,u(x))$ side by side. Right: a (3D)-plot showing the solution on the tadpole graph.}
    \label{fig:n1example}
\end{figure}

\subsection{A general flower graph}

For a general flower graph with loops of different half-lengths $\{ L_j \}_{j = 1}^N$, we introduce the following  parameters 
$$
p = \tilde{u}(L) = \tilde{u}_j(-L_j) = \tilde{u}_j(L_j), \quad 1 \leq j \leq N
$$
and 
$$
q_j = \tilde{u}'_j(-L_j) = -\tilde{u}_j'(L_j), \quad 1 \leq j \leq N.
$$
The NK conditions are satisfied if 
$$
q = \tilde{u}'(L) = 2 \sum_{j=1}^N q_j.
$$
The solution is defined by roots of the system of $(N+1)$ equations:
\begin{equation}
\label{period-system}
T\left(p,2 \sum_{j=1}^N q_j \right) = L, \quad T_0(p,q_j) = L_j, \quad 1 \leq j \leq N,
\end{equation}
where $T$ and $T_0$ are the two period functions introduced in (\ref{period-function}) and (\ref{period-function-second}), respectively.

We intend to prove that the mapping $(p,q_1,\dots,q_N) \to (L,L_1,\dots,L_N)$ is invertible so that 
for every given point $(L,L_1,\dots,L_N)$ in an open region $\Omega \subset \mathbb{R}^{N+1}_+$, there exists a unique 
solution for $(p,q_1,\dots,q_N) \in (0,1) \times (-\infty,0) \times \dots \times (-\infty,0)$. This gives the uniquely specified 
positive ground state in the existence region $\Omega$ defined by the invertible mapping 
$$
(0,1) \times (-\infty,0) \times \dots \times (-\infty,0) \ni (p,q_1,\dots,q_N) \to (L,L_1,\dots,L_N) \in \Omega \subset \mathbb{R}^{N+1}_+.
$$
The following theorem gives the result based on the monotonicity of the two period functions (\ref{period-function}) and (\ref{period-function-second}) in Lemmas 
\ref{lem-period-1} and \ref{lem-period-2}, respectively, 

\begin{theorem}
	\label{theorem-example-3}
	Let $\Gamma = [0,L] \times [-L_1,L_1] \times \dots \times [-L_N,L_N]$ be the flower graph with the Dirichlet condition at $x = 0$ and the NK condition at the interior vertex. Then, there is a simply connected region $\Omega \in \mathbb{R}^{N+1}_+$ such that the positive ground state exists for every $(L,L_1,\dots,L_N) \in \Omega$ and is unique. 
\end{theorem}

\begin{proof}
The Jacobian of the transformation  $(p,q_1,\dots,q_N) \to (L,L_1,\dots,L_N)$ is 
$$
J := \left| 
\begin{matrix} 
 \frac{\partial T}{\partial p} & 2 \frac{\partial T}{\partial q} & \dots & 2 \frac{\partial T}{\partial q} \\
\frac{\partial T_0^{(1)}}{\partial p} & \frac{\partial T_0^{(1)}}{\partial q_1} & \dots & 0 \\
\vdots & \vdots & \dots & \vdots \\
\frac{\partial T_0^{(N)}}{\partial p} & 0 & \dots & \frac{\partial T_0^{(N)}}{\partial q_N}
\end{matrix}
\right|,
$$
where $T_0^{(j)}$ denotes $T_0(p,q_j)$. Expanding the Jacobian yields 
\begin{equation}
    \label{Jacobian}
    J = \frac{\partial T}{\partial p} \prod_{j=1}^N \frac{\partial T_0^{(j)}}{\partial q_j} - 2 \frac{\partial T}{\partial q} 
    \sum_{j=1}^N \frac{\partial T_0^{(j)}}{\partial p} \prod_{k \neq j} \frac{\partial T_0^{(k)}}{\partial q_k}.
\end{equation}
We will prove that $J \neq 0$ and ${\rm sgn}(J) = (-1)^{N+1}$. By the monotonicity results (\ref{monotonicity}) and (\ref{monotonicity-second}), the first term in (\ref{Jacobian}) has the required sign. However, the second term in (\ref{Jacobian}) has the required sign for $p \in \left(0,\frac{1}{2}\right]$ due to (\ref{monotonicity-third}) but is generally inconclusive for $p \in \left(\frac{1}{2},1\right)$. To make it conclusive for every $N \geq 2$, we define again 
\begin{align*}
    \Delta &:= [E(p,q) + A(1)] \prod_{j=1}^N [E(p,q_j) + A(1)] J \\
    &= \left(-p(1-p) \mathcal{I} + q\right) \prod_{j=1}^N \left( q_j \mathcal{I}_j - \frac{(1-p)(1+2p)}{3p} \right) \\
    & \quad -2 \left( q \mathcal{I} + \frac{(1-p)(1+2p)}{3p} \right) \sum_{j=1}^N \left(-p(1-p) \mathcal{I}_j - q_j \right) 
    \prod_{k \neq j} \left( q_k \mathcal{I}_k -\frac{(1-p)(1+2p)}{3p} \right),
\end{align*}
where the positive integrals $\mathcal{I}$ and $\{ \mathcal{I}_j \}_{j=1}^N$ are defined similarly to (\ref{integrals}). We will now show that $\Delta$ can be written as a sum of terms, each is nonzero and has the sign of $(-1)^{N+1}$. 
This is definitely true for 
$$
-p(1-p) \mathcal{I} \prod_{j=1}^N \left( q_j \mathcal{I}_j - \frac{(1-p)(1+2p)}{3p} \right),
$$
$$
+2 p(1-p) \left( q \mathcal{I} + \frac{(1-p)(1+2p)}{3p} \right) \sum_{j=1}^N \mathcal{I}_j 
    \prod_{k \neq j} \left( q_k \mathcal{I}_k -\frac{(1-p)(1+2p)}{3p} \right),
$$
and 
$$
+2 q \mathcal{I} \sum_{j=1}^N q_j 
    \prod_{k \neq j} \left( q_k \mathcal{I}_k -\frac{(1-p)(1+2p)}{3p} \right)
$$
because $q \mathcal{I} + \frac{(1-p)(1+2p)}{3p} > 0$ and $q_j \mathcal{I}_j - \frac{(1-p)(1+2p)}{3p} < 0$ by 
(\ref{monotonicity}) and (\ref{monotonicity-second}), respectively, whereas $p \in (0,1)$ and $q \in (-\infty,0)$. 
The remaining terms are combined together as 
\begin{align*}
F(\hat{p}) := q  \prod_{j=1}^N \left( q_j \mathcal{I}_j - \hat{p} \right) + 2 \hat{p} \sum_{j=1}^N q_j 
    \prod_{k \neq j} \left( q_k \mathcal{I}_k - \hat{p} \right), \quad \hat{p} := \frac{(1-p)(1+2p)}{3p}, 
\end{align*}
where the second term has the wrong signs of $(-1)^N$. It is clear that $F$ is a polynomial in $\hat{p}$ with the highest term $\hat{p}^N$ 
being zero because 
$$
(-\hat{p})^N \left[ q - 2\sum_{j=1}^N q_j \right] = 0.
$$
Hence, $F$ is a polynomial of degree $N-1$ in $\hat{p}$ and we show that every coefficient of this polynomial 
is nonzero and has the sign of $(-1)^{N-1}$: 
\begin{align*}
    (-\hat{p})^{N-1} : & \quad q \sum_{j=1}^N q_j \mathcal{I}_j - 2 \sum_{j=1}^N q_j \sum_{k\neq j} q_k \mathcal{I}_k 
    = \sum_{j=1}^N q_j \mathcal{I}_j \left(q - 2 \sum_{k \neq j} q_k \right) > 0 \\
    (-\hat{p})^{N-2} : &\quad q \sum_{j=1}^N \sum_{k > j} q_j q_k \mathcal{I}_j \mathcal{I}_k 
    - 2 \sum_{j=1}^N q_j \sum_{k\neq j} q_k \mathcal{I}_k \sum_{m > k} q_m \mathcal{I}_m =  
    \sum_{j=1}^N \sum_{k > j} q_j q_k \mathcal{I}_j \mathcal{I}_k \left( q - \sum_{m\neq j,k} q_m \right) < 0 \\
    \vdots \qquad & \qquad \vdots \\
    (-\hat{p})^1 : & \quad q \sum_{j=1}^N \prod_{k \neq j} q_k \mathcal{I}_k - 2 \sum_{j=1}^N q_j \sum_{k\neq j} q_k \mathcal{I}_k 
    = \sum_{j=1}^N (q - 2 q_j) \prod_{k \neq j} q_k \mathcal{I}_k, \\
    (-\hat{p})^0 : & \quad q \prod_{j=1}^N q_k \mathcal{I}_k,
\end{align*}
where the last coefficient has the sign of $(-1)^{N+1}$ and the coefficient for $(-\hat{p})^1$ has the sign of $(-1)^N$. 
Hence, every coefficient of $F(\hat{p})$ has the sign of $(-1)^{N+1}$, whereas $\hat{p} > 0$. This proves that 
$\Delta$ is nonzero and has the sign of $(-1)^{N+1}$.
\end{proof}

\begin{remark}
    The lower boundary of the simply connected region $\Omega \in \mathbb{R}^{N+1}_+$ is given by the condition $\lambda_0(\Gamma) = 1$, where $\lambda_0(\Gamma)$ is the lowest eigenvalue of $-\Delta_{\Gamma}$ in $L^2(\Gamma)$. The solution of $-\Delta_{\Gamma} \Psi = \lambda_0(\Gamma) \Psi$ in Remark \ref{rem-boundary} is generalized as 
    $$
    \psi(x) = \sin(\sqrt{\lambda_0(\Gamma)} x), \quad x \in (0,L)
    $$
  and
$$
\psi_j(x) = \frac{\sin(\sqrt{\lambda_0(\Gamma)} L)}{\cos(\sqrt{\lambda_0(\Gamma)} L_j)} \cos(\sqrt{\lambda_0(\Gamma)} x), \quad x \in (-L_j,L_j),  \quad j = 1, \dots, N. 
$$
The NK condition is satisfied if and only if $\lambda_0(\Gamma)$ is found from the transcendental equation 
$$
2 \sum_{j=1}^N \tan(\sqrt{\lambda_0(\Gamma)} L_j) = \cot(\sqrt{\lambda_0(\Gamma)} L).
$$
Since $\lambda_0(\Gamma)$ is monotonically decreasing with respect to $(L,L_1,\dots,L_N)$ by Lemma \ref{lem-linear-1}, the lower boundary of the existence region $\Omega$ can be parameterized as $L = L(L_1,\dots,L_N)$ given by 
\begin{equation}
    \label{lower-boundary-general}
L(L_1,L_2,\dots,L_N) = \arccot \left( 2 \sum_{j=1}^N \tan(L_j) \right).
\end{equation}
The surface for the lower boundary of $\Omega$ for $N = 2$ is shown in Figure \ref{fig:3dRegion}.
\end{remark}

\begin{figure}[htb!]
    \centering
    \includegraphics[width=6cm,height=4cm]{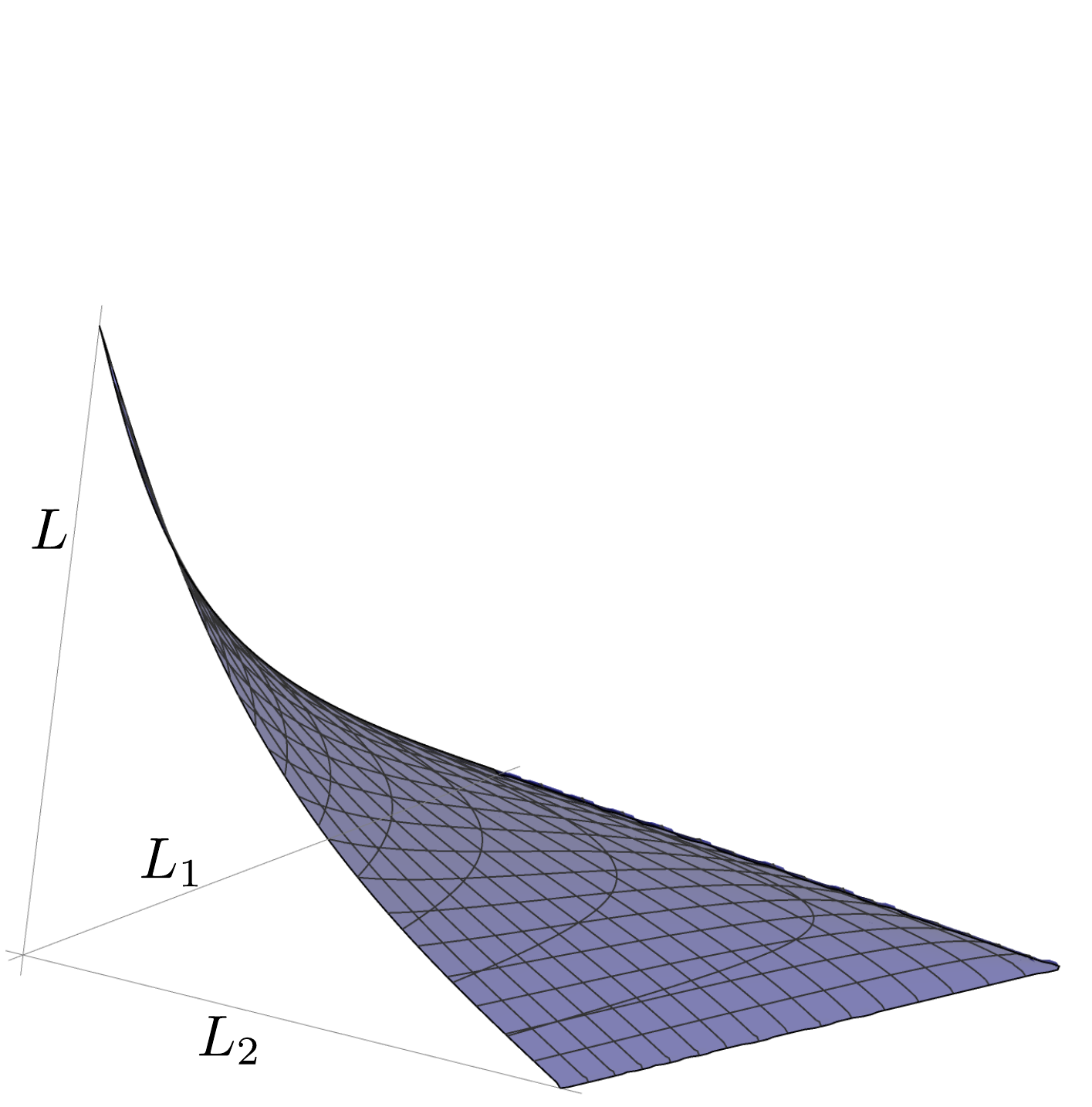}
    \caption{A plot of the lower boundary (\ref{lower-boundary-general}) for $N = 2$ in the $(L_1,L_2,L)$ space. The corner points are 
    $\left(\frac{\pi}{2},0,0\right)$,     $\left(0,\frac{\pi}{2},0\right)$, and     $\left(0,0,\frac{\pi}{2} \right)$. }
    \label{fig:3dRegion}
\end{figure}

Figure \ref{fig:n2flower} illustrates the construction of the positive ground state on the flower graph with $N = 2$ with $L_1 \neq L_2$ 
by using parts of three integral curves of the second-order equation (\ref{ode-tilde}). For the parameter values of $(L_1,L_2,L)$, 
a part of the integral curve for the pendant lies outside the homoclinic orbit, see Remark \ref{rem-outside}.

\begin{figure}[htb!]
    \centering
    \includegraphics[width=5cm,height=4cm]{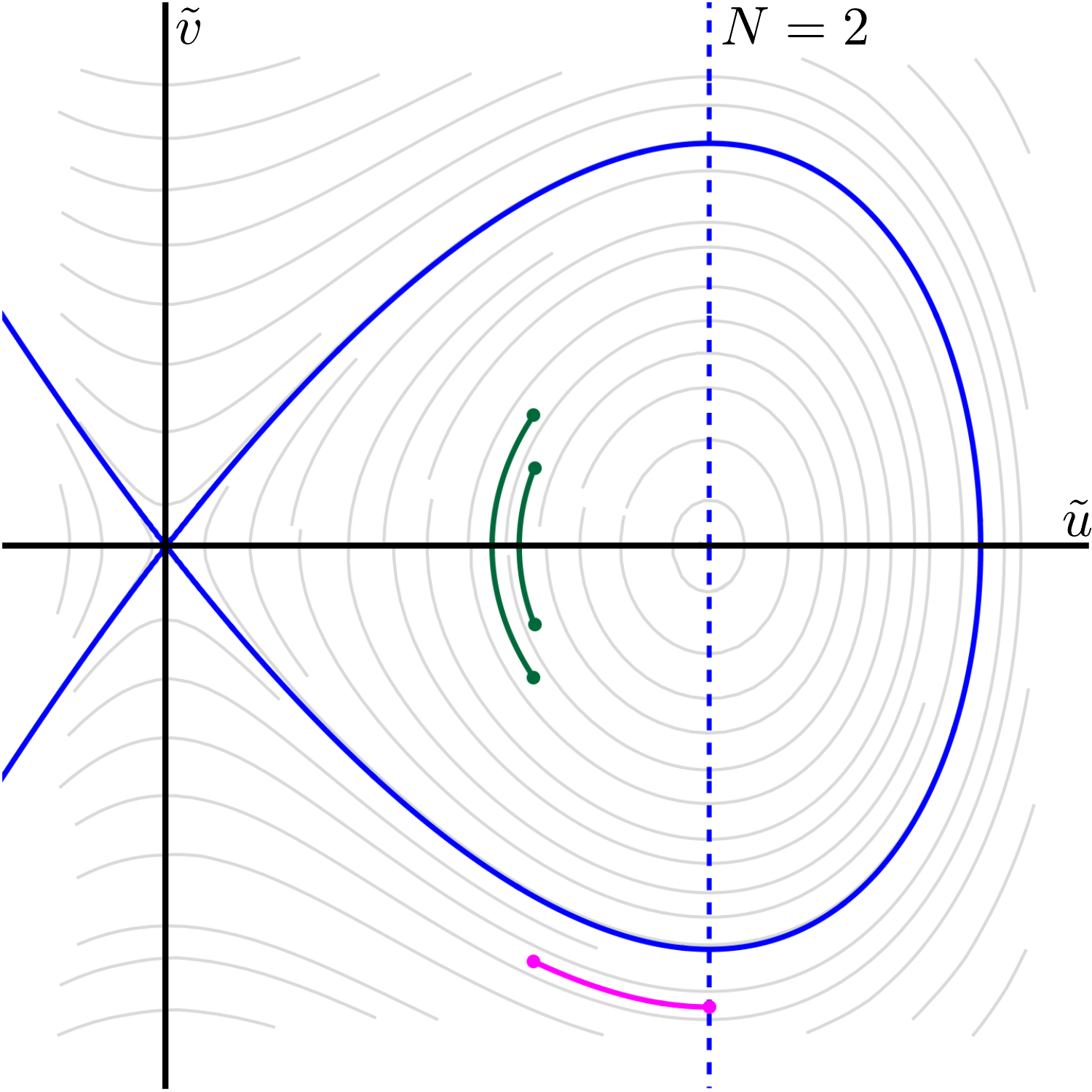}
    \includegraphics[width=5cm,height=4cm]{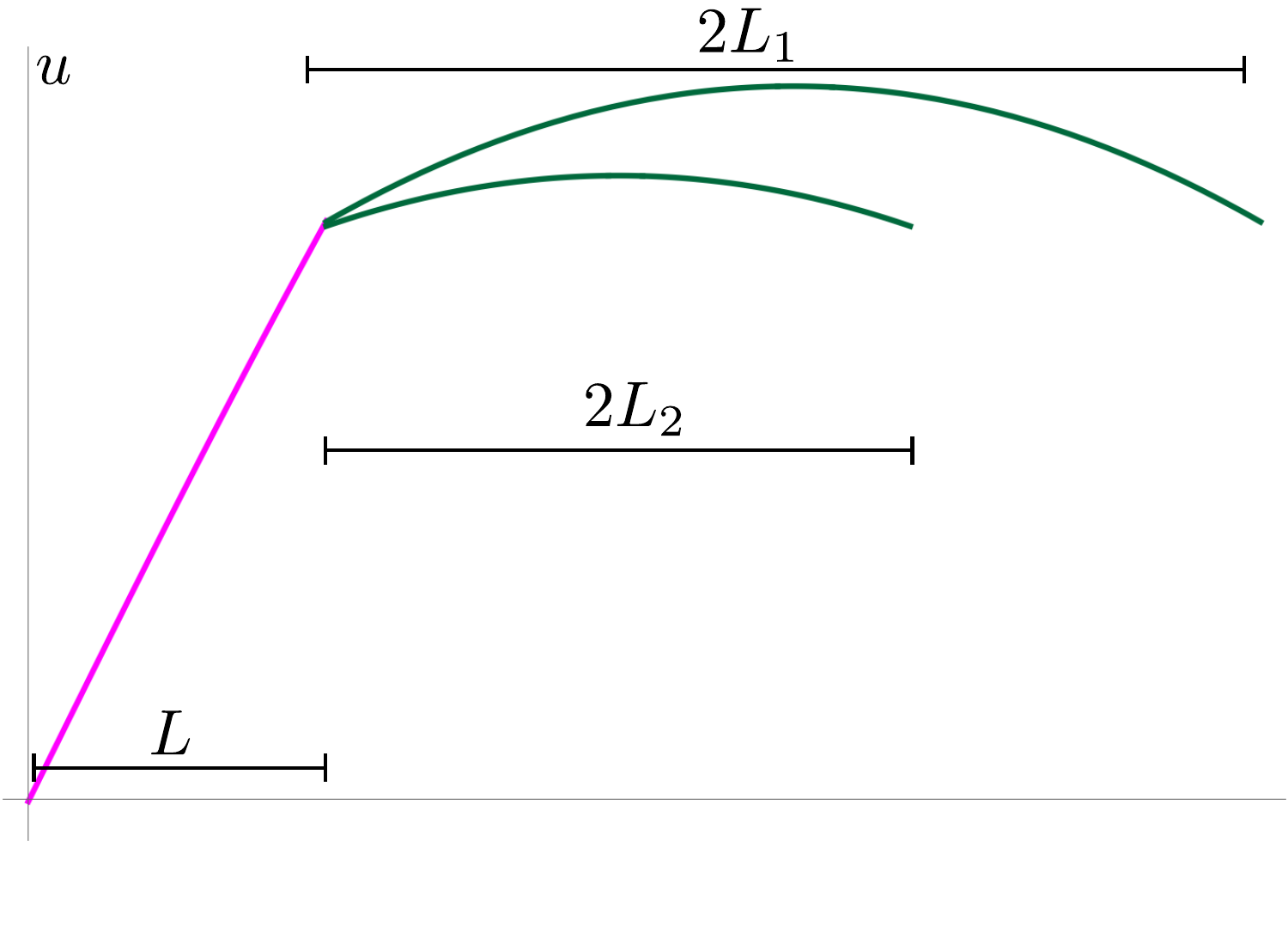}
    \includegraphics[width=5cm,height=4cm]{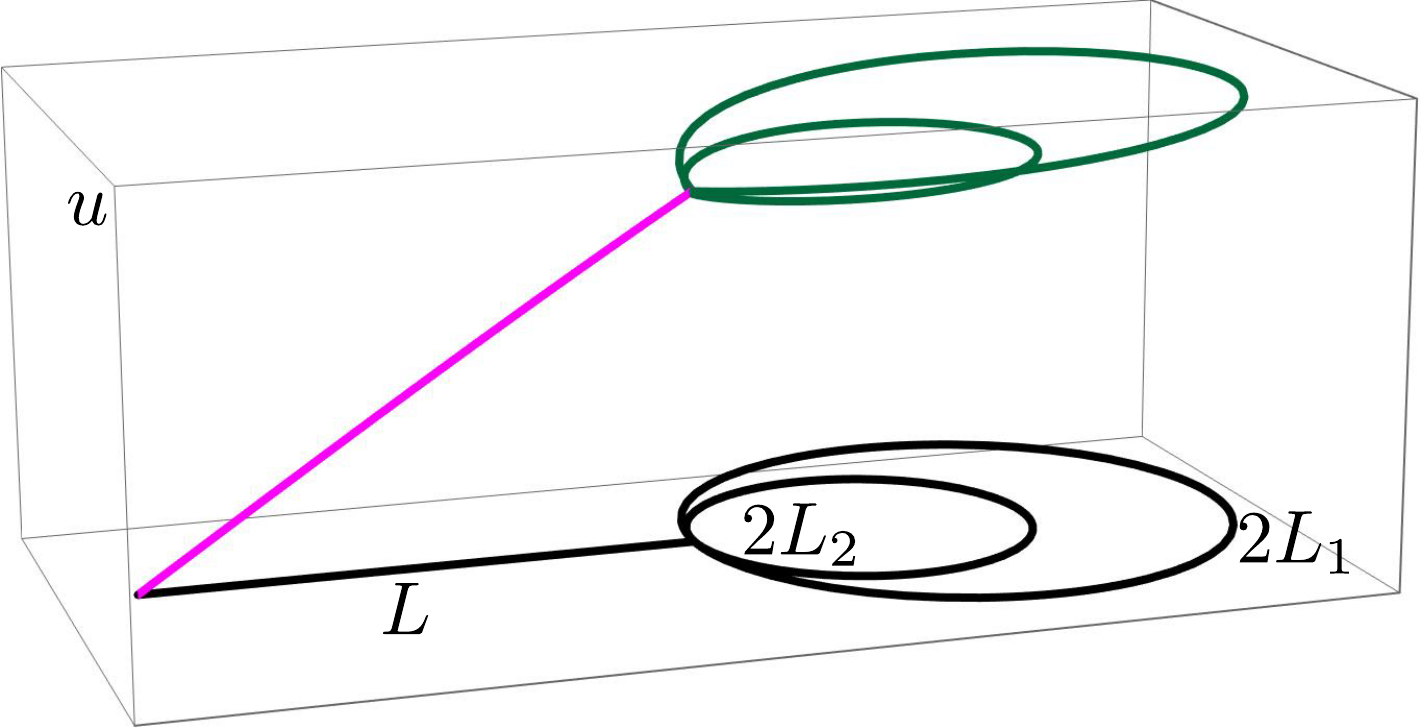}
    \caption{The positive ground state on a flower graph with a stem of length $L = 0.51$ and two loops of half-lengths $L_1 = 0.8$ and $L_2=0.5$. Left: parts of three integral curves on the phase plane $(\tilde{u},\tilde{v})$. Center: a plot in variables $(x,u(x))$ side by side. Right: a (3D)-plot showing the solution on the flower graph.}
    \label{fig:n2flower}
\end{figure}

\section*{Acknowlegements} A part of this work was performed during the visit of D. E. Pelinovsky to the University of Sydney. Both authors would like to acknowledge support of the Australian Research Council under grant ARC DP210101102. RM would like to acknowledge Jeremy Marzuola for fruitful discussions.

\nocite{DGK25}
\bibliographystyle{alpha}
\bibliography{FKPPGraph} 

\end{document}